\documentclass[12pt]{amsart}

\usepackage{amsfonts}
\usepackage{amsmath}
\usepackage{amssymb}
\usepackage{mathrsfs}
\usepackage{amscd}
\usepackage[all]{xy}
\usepackage[pdftex,final]{graphicx}
\input cyracc.def

\newtheorem{lem}{Lemma}[section]
\newtheorem{thm}[lem]{Theorem}

\newtheorem{cor}[lem]{Corollary}
\theoremstyle{definition}
\newtheorem{defi}[lem]{Definition}
\newtheorem{exa}[lem]{Example}
\newtheorem{rem}[lem]{Remark}

\title{Homology Stabilty for the special linear group of a field and Milnor-Witt $K$-theory}
\author{Kevin Hutchinson,  \   Liqun Tao}
\address{School of Mathematical Sciences,
 University College Dublin}
\email{kevin.hutchinson@ucd.ie, \ liqun.tao@ucd.ie}
\date{\today}


\newcommand{\imp}{\Longrightarrow}
\renewcommand{\iff}{\Longleftrightarrow}

\newcommand{\Q}{\Bbb{Q}}

\newcommand{\Z}{\Bbb{Z}}
\newcommand{\N}{\Bbb{N}}

\newcommand{\an}[1]{\langle{#1}\rangle}

\newcommand{\ee}[1]{\epsilon_{#1}}

\newcommand{\kk}[1]{\kappa_{#1}}
\newcommand{\mm}[1]{\mu_{#1}}

\newcommand{\xgen}[2]{X_{#1}(#2)}

\newcommand{\xgenp}[2]{\bar{X}_{#1}(#2)}
\newcommand{\cgen}[2]{C_{#1}(#2)}

\newcommand{\SSn}[1]{\tilde{S}(#1)}
\newcommand{\SSp}[1]{\tilde{S}(#1)^+}
\newcommand{\SSd}[1]{\tilde{S}(#1)^{\mbox{\tiny dec}}}
\newcommand{\SSf}[2]{\tilde{\Sigma}_{#1}(#2)}
\newcommand{\SSind}[1]{\tilde{S}(#1)^{\mbox{\tiny ind}}}
\newcommand{\hgenp}[2]{\bar{H}_{#1}(#2)}
\newcommand{\sgenp}[2]{\bar{S}_{#1}(#2)}
\newcommand{\ssb}[1]{\lfloor #1 \rceil}
\newcommand{\ssf}[1]{\big[ #1\big] }
\newcommand{\ksb}[1]{\{\{ #1\}\}}
\newcommand{\ksp}[1]{[#1]}
\newcommand{\hsp}[1]{\langle #1\rangle}
\newcommand{\ssbp}[1]{[[#1]]}
\newcommand{\ssprod}{\ast}
\newcommand{\psprod}{\circledast}
\newcommand{\compt}[2]{\mathcal{C}_{#1}^{\tau}(#2)}
\newcommand{\cgenp}[2]{\bar{C}_{#1}(#2)}

\newcommand{\SSS}{\mathrm{\tilde{S}}}
\newcommand{\SH}[2]{\mathrm{SH}_{#1}(#2)}
\newcommand{\T}[1]{T_{#1}}

\newcommand{\Dd}[1]{D_{#1}}

\newcommand{\Aa}[1]{\mathcal{A}_{#1}}

\newcommand{\am}{\mathcal{AM}}
\newcommand{\add}{additive\ }
\newcommand{\mult}{multiplicative\ }
\newcommand{\ad}[1]{{#1}_{\mathcal{A}}}
\newcommand{\mul}[1]{{#1}_{\mathcal{M}}}
\newcommand{\EEE}[1]{\mathcal{E}(#1)}
\newcommand{\EEEz}[1]{\mathcal{E}^0(#1)}
\newcommand{\EEEp}[1]{\mathcal{E}^+(#1)}
\newcommand{\EE}[2]{\EEE{#1,#2}}
\newcommand{\EEz}[2]{\EEEz{#1,#2}}
\newcommand{\EEp}[2]{\EEEp{#1,#2}}

\newcommand{\fil}[2]{\mathcal{F}_{#2,#1}}

\renewcommand{\dim}[2]{\mathrm{dim}_{#1}\left(#2\right)}
\newcommand{\<}{\langle}
\renewcommand{\>}{\rangle}
\newcommand{\il}[1]{\< #1 \>}
\newcommand{\card}[1]{\left| #1 \right|}

\renewcommand{\ker}[1]{\mathrm{Ker}(#1)}
\newcommand{\image}[1]{\mathrm{Im}(#1)}

\newcommand{\coker}[1]{\mathrm{Coker}(#1)}

\renewcommand{\hom}[3]{\mathrm{Hom}_{#1}(#2,#3)}

\newcommand{\gr}[2]{#1[#2]}
\newcommand{\GR}[2]{#1\left[ #2\right]}
\newcommand{\aug}[1]{\mathcal{I}_{#1}}

\newcommand{\grf}[1]{\gr{\Z}{#1^\times}}

\newcommand{\gw}[1]{\mathrm{GW}(#1)}
\newcommand{\wgen}[1]{\langle #1 \rangle}
\newcommand{\pfist}[1]{\langle\langle #1 \rangle\rangle}

\newcommand{\ffist}[1]{\langle\langle #1 \rangle\rangle}
\newcommand{\fgen}[1]{\left\langle #1 \right\rangle}

\newcommand{\filt}[2]{F_{#1}{#2}}
\newcommand{\grd}[2]{\mathrm{gr}_{#1}#2}

\newcommand{\ext}[2]{\bigwedge^{#1}\left({#2}\right)}
\newcommand{\tens}[3]{\mathrm{T}^{#1}_{#2}(#3)}

\newcommand{\extr}[3]{\bigwedge^{#1}_{#2}(#3)}
\newcommand{\proj}[2]{\mathbb{P}^{#1}(#2)}
\newcommand{\p}[1]{\overline{#1}}

\newcommand{\id}[1]{\mathrm{Id}_{#1}}

\newcommand{\aff}[2]{\mathrm{A}(#1,#2)}
\newcommand{\saff}[2]{\mathrm{SA}(#1,#2)}

\newcommand{\diag}[1]{\mathrm{diag}(#1)}

\newcommand{\shift}[2]{{#2}[#1]}

\newcommand{\sgn}[1]{\mathrm{sgn}(#1)}

\newcommand{\mset}[1]{\mathcal{S}_{#1}}

\newcommand{\msetp}[1]{\mathcal{S}^+_{#1}}
\newcommand{\mwk}[2]{K^{\mathrm{\small MW}}_{#1}({#2})}

\newcommand{\milk}[2]{K^{\mathrm{\small M}}_{#1}({#2})}
\newcommand{\milkt}[2]{k^{\mathrm{\small M}}_{#1}({#2})}

\newcommand{\homol}[2]{\mathrm{H}_{#1}(#2)}
\newcommand{\ho}[3]{\mathrm{H}_{#1}(#2,#3 )}
\newcommand{\hoz}[2]{\mathrm{H}_{#1}(#2, \Z)}

\newcommand{\cores}[2]{\mathrm{cor}^{#2}_{#1}}

\newcommand{\genl}[2]{\mathrm{GL}_{#1}(#2)}
\newcommand{\specl}[2]{\mathrm{SL}_{#1}(#2)}
\newcommand{\gnl}[1]{\mathrm{GL}(#1)}
\newcommand{\spcl}[1]{\mathrm{SL}(#1)}

\begin{document}

\maketitle

\begin{abstract}
Let $F$ be a field of characteristic zero and let 
$f_{t,n}$ be the stabilization homomorphism $\hoz{n}{\specl{t}{F}}\to\hoz{n}{\specl{t+1}{F}}$. We prove the following results:
For all $n$, $f_{t,n}$ is an isomorphism if $t\geq n+1$ and is surjective for $t=n$, confirming a conjecture of C-H. Sah. 
$f_{n,n}$ is an isomorphism when $n$ is odd and when $n$ is even the kernel is isomorphic to $I^{n+1}(F)$, the $(n+1)$st power of the 
fundamental ideal of the Witt Ring of $F$. When $n$ is even the cokernel of $f_{n-1,n}$ is isomorphic to $\mwk{n}{F}$, the $n$th 
Milnor-Witt $K$-theory group of $F$. When $n$ is odd, the cokernel of $f_{n-1,n}$ is isomorphic to $2\milk{n}{F}$, where $\milk{n}{F}$ 
is the $n$th Milnor $K$-group of $F$.
\end{abstract}
\section{Introduction}
%Homology Stability background
Given a family of groups $\{ G_t\}_{t\in\N}$ 
with canonical homomorphisms 
$G_t\to G_{t+1}$, we say that the family has homology stability if there exist constants $K(n)$ such that the 
natural maps $\hoz{n}{G_t}\to\hoz{n}{G_{t+1}}$ are isomorphisms for $t\geq K(n)$. The question of homology stability for families of 
linear groups over a ring $R$ - general linear groups, special linear groups, symplectic, orthogonal and unitary groups - 
has been studied since the 1970s in connection with applications to algebraic $K$-theory, algebraic topology, the scissors congruence 
problem, and the homology of Lie groups. These families of linear groups are known to have homology stability at least when the rings 
satisfy some appropriate finiteness condition, and in particular in the case of fields and local rings 
(\cite{charney:stab},\cite{vogtmann:stab},\cite{vogtmann:stab2},\cite{vanderkallen:homstab}, \cite{guin:stab},\cite{betley:stab}, 
\cite{panin:stab},\cite{mirzaii:stab},\cite{mirzaii:stab2}).  
It seems to be a delicate - but 
interesting and apparently important - question, however, to decide the minimal possible value of $K(n)$ for a particular class of linear 
groups (with coefficients in a given class of rings) 
and the nature of the obstruction to extending the stability range further. 

The best illustration of this last remark are the results of Suslin on the integral homology of the general linear group of 
a field in the paper \cite{sus:homgln}.  He proved that, for an infinite field $F$,
the maps $\hoz{n}{\genl{t}{F}}\to\hoz{n}{\genl{t+1}{F}}$ are isomorphisms for $t\geq n$ (so that $K(n)=n$ in this case), while 
the cokernel of the map  $\hoz{n}{\genl{n-1}{F}}\to\hoz{n}{\genl{n}{F}}$ is naturally isomorphic to the $n$th Milnor $K$-group,
$\milk{n}{F}$.In fact, if we let 
\[
H_n(F):= \coker{\hoz{n}{\genl{n-1}{F}}\to\hoz{n}{\genl{n}{F}}},
\]
his arguments show that there is an isomorphism of graded rings $H_\bullet(F)\cong \milk{\bullet}{F}$ (where the multiplication on the 
first term comes from direct sum of matrices and cross product on homology). In particular, the non-negatively graded ring $H_\bullet(F)$
is generated in dimension $1$.

Recent work of Barge and Morel (\cite{barge:morel}) suggested that Milnor-Witt $K$-theory 
may play a somewhat analogous role for the homology
of the special linear group. The Milnor-Witt  $K$-theory of $F$ is a $\Z$-graded ring $\mwk{\bullet}{F}$ surjecting naturally onto 
Milnor $K$-theory. It arises as a ring of operations in stable motivic homotopy theory. (For a definition see section \ref{sec:background}
below, and for more details see \cite{morel:trieste,morel:puiss,morel:a1}.)
  Let $SH_n(F):=\coker{\hoz{n}{\specl{n-1}{F}}\to\hoz{n}{\specl{n}{F}}}$ for $n\geq 1$, and let $SH_0(F)=\GR{\Z}{F^\times}$
for convenience.
Barge and Morel construct a map of graded algebras $SH_\bullet(F)\to\mwk{\bullet}{F}$ for which the square
\begin{eqnarray*}
\xymatrix{
SH_\bullet(F)\ar[r]\ar[d]
&
\mwk{\bullet}{F}\ar[d]\\
H_\bullet(F)\ar[r]
&
\milk{\bullet}{F}
}
\end{eqnarray*}
commutes.

A result of Suslin (\cite{sus:tors}) implies that the map $\hoz{2}{\specl{2}{F}}=SH_2(F)\to\mwk{2}{F}$ is an isomorphism. 
Since positive-dimensional Milnor-Witt $K$-theory is generated by elements of degree $1$, it follows that 
the map of Barge and Morel is surjective in even dimensions  
greater than or equal to $2$. They ask the question whether it is in 
fact an isomorphism in even dimensions.

As to the question of the range of homology stability for the special linear groups of an infinite field, as far as the authors are aware 
the most general result to date 
is still that of van der Kallen \cite{vanderkallen:homstab}, whose results apply to much more general classes of rings. In the case of a 
field, he proves homology stability for $\hoz{n}{\specl{t}{F}}$ in the range $t\geq 2n+1$. On the other hand, known results when 
$n$ is small suggest a much larger range. For example, the theorems of Matsumoto and Moore imply that the maps 
$\hoz{2}{\specl{t}{F}}\to\hoz{2}{\specl{t+1}{F}}$ are isomorphisms for $t\geq 3$ and are surjective for $t=2$. In the paper 
\cite{sah:discrete3} (Conjecture 2.6), C-H. Sah conjectured that for an infinite field $F$ 
(and more generally for a division algebra with infinite centre),  the homomorphism $\hoz{n}{\specl{t}{F}}\to
\ho{n}{\specl{t+1}{F}}$ is an isomorphism if $t\geq n+1$ and is surjective for $t=n$. 

The present paper addresses the above questions of Barge/Morel and Sah in the case of a field of characteristic zero. We prove the 
following results about the homology stability for special linear groups:

\begin{thm}\label{thm:homstab} Let $F$ be a field of characteristic $0$. For $n,t\geq 1$, 
let $f_{t,n}$ be the 
stabilization homomorphism $\hoz{n}{\specl{t}{F}}\to\hoz{n}{\specl{t+1}{F}}$
\begin{enumerate}
\item $f_{t,n}$ is an isomorphism for $t\geq n+1$ and is surjective for $t=n$.
\item If $n$ is odd $f_{n,n}$ is an isomorphism
\item If $n$ is even the kernel of $f_{n,n}$ is isomorphic to $I^{n+1}(F)$.
\item For even $n$ the cokernel of $f_{n-1,n}$ is naturally isomorphic to $\mwk{n}{F}$.
\item For odd $n\geq 3$ the cokernel of $f_{n-1,n}$ is naturally isomorphic to $2\milk{n}{F}$. 
\end{enumerate}
\end{thm}
\begin{proof} The proofs of these statements can be found below as follows:
\begin{enumerate}
\item Corollary \ref{cor:stab}.  
\item Corollary \ref{cor:odd}.
\item Corollary \ref{cor:last}.
\item Corollary \ref{cor:even}.
\item Corollary \ref{cor:last}
\end{enumerate}
\end{proof}

Our strategy is to adapt Suslin's argument for the general linear group 
in \cite{sus:homgln} to the case of the special linear group. Suslin's  argument 
is an ingenious variation on the method of van der Kallen in \cite{vanderkallen:homstab}, in turn based on ideas of Quillen. The broad 
idea is to find a highly connected simplicial complex on which the group $G_t$ acts and for which the stabilizers of simplices are 
(approximately) the groups $G_r$, with $r\leq t$, and then to use this to construct a spectral sequence calculating the homology 
of the $G_n$ in terms of the homology of the $G_r$.  Suslin constructs a family $\EEE{n}$ of such spectral sequences, calculating the 
homology of $\genl{n}{F}$. He constructs partially-defined products $\EEE{n}\times\EEE{m}\to \EEE{n+m}$ and then proves some periodicity 
and decomposabilty properties which allow him to conclude  by an easy induction.

 Initially, the attempt to extend these arguments to the case of $\specl{n}{F}$ does 
not appear very promising. Two obstacles to extending Suslin's arguments become quickly apparent. 

The main obstacle is 
Suslin's Theorem 1.8 which says that a certain inclusion of a block diagonal linear group in a block triangular group is a homology 
isomorphism. The corresponding statement for subgroups of the special linear group are emphatically false, as elementary calculations 
easily show. Much of Suslin's subsequent results - in particular, the periodicity and decomposability properties of the spectral 
sequences $\EEE{n}$ and of the graded algebra 
$S_\bullet(F)$ which plays a central role - depend on this theorem. And, indeed, the analogous spectral sequences and 
graded algebra 
which arise when we replace the general linear with the special linear group  do not have these periodicity and decomposability 
properties.   

However, it turns out - at least when the characteristic is zero - that the failure of Suslin's Theorem 1.8 is not fatal. A crucial 
additional structure is available to us in the case of the special linear group; almost everything in sight in a $\grf{F}$-module. In the 
analogue of Theorem 1.8, the map of homology groups is a split inclusion whose cokernel has a completely different character as 
a $\grf{F}$-module than the  homology of the block diagonal group. The former is `\add', while the latter is `\mult', notions which we 
define and explore in section \ref{sec:am} below. This leads us to introduce 
the concept of `$\am$ modules', which decompose in a canonical way 
into a direct sum of an \add factor and a \mult factor. This decomposition is sufficiently canonical that in our graded ring structures 
the \add and \mult parts are each ideals. By working modulo the messy \add factors and projecting onto \mult parts, we recover 
an analogue of Suslin's Theorem 1.8 (Theorem \ref{thm:speclam} below), which we then use to prove the necessary periodicity 
(Theorem \ref{thm:main}) and decomposability (Theorem \ref{thm:ind}) results. 

A second obstacle to emulating the case of the general linear group is the vanishing of the groups $\hoz{1}{\specl{n}{F}}$. The algebra 
$H_\bullet(F)$, according to Suslin's arguments, is generated by degree $1$. On the other hand,
 $SH_1(F)=0=\hoz{1}{\specl{1}{F}}=0$. This means that the best we can 
hope for in the case of the special linear group is that the algebra $SH_\bullet(F)$ is generated by degrees $2$ and $3$. This indeed 
turns out to be essentially the case, but it means we have to work harder to get our induction off the ground. The necessary arguments 
in degree $n=2$ amount to the Theorem of Matsumoto and Moore, as well as variations due to Suslin (\cite{sus:tors}) and Mazzoleni
(\cite{mazz:sus}).  The argument in degree $n=3$ was supplied recently in a paper by the present authors (\cite{hutchinson:tao2}).

We make some remarks on the hypothesis of  characteristic zero in this paper: This assumption is used in our definition of $\am$-modules
and the derivation of their properties in section \ref{sec:am} below. In fact, a careful reading of the proofs in that section will show 
that at any given point all that is required is that the prime subfield be sufficiently large; it must contain an element of order not
dividing $m$ for some appropriate $m$. Thus in fact our arguments can easily be adapted to show that 
our main results on homology stability for the $n$th homology group of the 
special linear groups are true provided the prime field  is sufficiently large (in a way that depends on $n$). 
However, we have not attempted  here to make this more explicit. To do so would make the statements of the results unappealingly 
complicated, and we will leave it instead to a later paper to deal with the case of positive characteristic. We  believe that 
an appropriate extension of the notion of $\am$-module will unlock the characteristic $p>0$ case.

As to our restriction to fields rather than more general rings, 
we note that Daniel Guin \cite{guin:stab} has extended Suslin's results to a larger class of rings 
with many units. We have not yet investigated a similar extension of the results below to this larger class of rings.    

\section{Notation and Background Results}\label{sec:background}
\subsection{Group Rings and Grothendieck-Witt Rings}
For a group $G$, we let $\GR{\Z}{G}$ denote the corresponding integral group ring. It has an additive 
$\Z$-basis consisting of 
the elements $g\in G$, and is made into a ring by linearly extending  the  multiplication of group elements. 
In the case that the group $G$ is the multiplicative group, $F^\times$, of a field $F$, we will denote the basis elements by 
$\fgen{a}$, for $a\in F^\times$. We use this notation in order, for example, to distinguish the elements 
$\fgen{1-a}$ from $1-\fgen{a}$, or $\fgen{-a}$ from $-\fgen{a}$, and also because it coincides, conveniently for our purposes, with 
the notation for generators of the Grothendieck-Witt ring (see below).
There is an augmentation homomorphism $\epsilon:
\GR{\Z}{G}\to
\Z$, $\fgen{g}\mapsto 1$, whose kernel is the augmentation ideal $\aug{G}$, generated by the elements 
$g-1$. Again, if $G= F^\times$, we denote these generators by $\ffist{a}:=\fgen{a}-1$.    

The Grothendieck-Witt ring of a field $F$ is the Grothendieck group, $\gw{F}$, of the set of isometry classes of nondgenerate symmetric 
bilinear forms under orthogonal sum. Tensor product of forms induces a natural multiplication on the group. As an abstract ring, 
this can be described as the quotient of the ring $\GR{\Z}{F^\times/(F^\times)^2}$ by the ideal generated by the elements 
$\ffist{a}\cdot\ffist{1-a}$, $a\not= 0,1$. (This is just a mild reformulation of the presentation given in Lam, \cite{lam:intro}, 
Chapter II,
 Theorem 4.1.) Here, the induced ring homomorphism $\GR{\Z}{F^\times}\to \GR{\Z}{F^\times/(F^\times)^2}\to
\gw{F}$, sends $\fgen{a}$ to the class of the $1$-dimensional form with matrix $[a]$. This class is (also) denoted $\wgen{a}$. $\gw{F}$
is again an augmented ring and the augmentation ideal, $I(F)$, - also called the \emph{fundamental ideal} - 
is generated by \emph{Pfister $1$-forms}, $\pfist{a}$. It follows that the 
$n$-th power, $I^n(F)$, of this ideal is generated by \emph{Pfister} $n$-forms $\pfist{a_1,\ldots,a_n}:=\pfist{a_1}\cdots\pfist{a_n}$.

Now let $\textbf{h}:=\wgen{1}+\wgen{-1}=\pfist{-1}+2\in \gw{F}$. 
Then $\textbf{h}\cdot I(F)=0$, and the \emph{Witt ring} of $F$ is the ring
\[
W(F):=\frac{\gw{F}}{\langle \textbf{h}\rangle}=\frac{\gw{F}}{\textbf{h}\cdot \Z}.
\] 
Since $\textbf{h}\mapsto 2$ under the augmentation, there is a natural ring homomorphism $W(F)\to \Z/2$. The fundamental ideal $I(F)$ of 
$\gw{F}$ 
maps isomorphically to the kernel of this ring homomorphism under the map $\gw{F}\to W(F)$, and we also let $I(F)$ denote this ideal.

For $n\leq 0$, we define $I^n(F):=W(F)$. The graded additive group $I^\bullet(F)=\{ I^n(F)\}_{n\in\Z}$ is given the structure of a 
commutative graded ring using the natural graded multiplication induced from the multiplication on $W(F)$. In particular, if we let 
$\eta\in I^{-1}(F)$ be the element corresponding to $1\in W(F)$, then multiplication by $\eta:I^{n+1}(F)\to I^n(F)$ is just the natural 
inclusion.

\subsection{Milnor $K$-theory and Milnor-Witt $K$-theory}

The Milnor ring of a field $F$  (see \cite{milnor:intro}) is the graded ring
\(
\milk{\bullet}{F}%:=\frac{(F^\times)^{\otimes_\Z \bullet}}{\langle a\otimes(1-a)| a\not=0,1\rangle}.
\)
with the following presentation:

Generators: $\{ a\}$ , $a\in F^\times$, in dimension $1$.

Relations:
\begin{enumerate}
\item[(a)] $\{ ab\} =\{ a\} +\{ b\}$ for all $a,b\in F^\times$.
\item[(b)] $\{a\}\cdot\{1-a\}=0$ for all $a\in F^\times\setminus\{ 1\}$.
\end{enumerate}
The product $\{ a_1\}\cdots\{ a_n\}$ in $\milk{n}{F}$ is also written $\{ a_1,\ldots,a_n\}$. So $\milk{0}{F}=\Z$ and $\milk{1}{F}$ is 
an additive group isomorphic to $F^\times$.

We let $\milkt{\bullet}{F}$ denote the graded ring $\milk{\bullet}{F}/2$ and let $i^n(F):=I^n(F)/I^{n+1}(F)$, so that $i^\bullet(F)$ is a 
non-negatively graded ring. 

In the 1990s, Voevodsky and his collaborators proved a fundamental and deep theorem - originally conjectured by Milnor 
(\cite{milnor:quad}) - 
relating Milnor $K$-theory to quadratic form theory:

\begin{thm}[\cite{voevodsky:orlovvishik}]
There is a natural isomorphism of graded rings $\milkt{\bullet}{F}\cong i^\bullet(F)$ sending $\{ a\}$ to $\pfist{a}$.

In particular for all $n\geq 1$ we have a natural identification of $\milkt{n}{F}$ and $i^n(F)$ under which the symbol 
$\{ a_1,\ldots,a_n\}$ 
corresponds to the class of the form $\pfist{a_1,\ldots,a_n}$.
\end{thm}

The Milnor-Witt $K$-theory of a field is the graded ring $\mwk{\bullet}{F}$ with the following presentation 
(due to F. Morel and M. Hopkins,
see \cite{morel:trieste}):

Generators: $[a]$, $a\in F^\times$, in dimension $1$ and a further generator $\eta$ in dimension $-1$.

Relations: 
\begin{enumerate}
\item[(a)] $[ab]=[a]+[b]+\eta\cdot [a]\cdot [b]$ for all $a,b\in F^\times$
\item[(b)] $[a]\cdot[1-a]=0$ for all $a\in F^\times\setminus\{ 1\}$
\item[(c)] $\eta\cdot [a]=[a]\cdot \eta$ for all $a\in F^\times$
\item[(d)] $\eta\cdot h=0$, where $h=\eta\cdot [-1] +2\in \mwk{0}{F}$.
\end{enumerate}

Clearly there is a unique surjective homomorphism of graded rings $\mwk{\bullet}{F}\to\milk{\bullet}{F}$ sending $[a]$ to $\{ a\}$
 and inducing an isomorphism
\[
\frac{\mwk{\bullet}{F}}{\langle \eta \rangle}\cong \milk{\bullet}{F}.
\]

Furthermore, there is a natural surjective homomorphism of graded rings $\mwk{\bullet}{F}\to I^\bullet(F)$ sending $[a]$ to $\pfist{a}$ and 
$\eta$ to $\eta$. Morel shows that there is an induced isomorphism of graded rings
\[
\frac{\mwk{\bullet}{F}}{\langle h \rangle}\cong I^{\bullet}(F).
\]

%Note that the composite $\mwk{\bullet}{F}\to I^\bullet(F)\to i^\bullet(F)$ sends $[a_1]\cdots [a_n]$ to the class of 
%$\pfist{a_1,\ldots, a_n}$. 

The main structure theorem on Milnor-Witt $K$-theory is the following theorem of Morel:

\begin{thm}[Morel, \cite{morel:puiss}]
The commutative square of graded rings
\begin{eqnarray*}
\xymatrix{
\mwk{\bullet}{F}\ar[r]\ar[d]
&
\milk{\bullet}{F}\ar[d]\\
I^\bullet(F)\ar[r]
&
\milkt{\bullet}{F}
}
\end{eqnarray*}
is cartesian.
\end{thm}

Thus for each $n\in \Z$ we have an isomorphism
\[
\mwk{n}{F}\cong \milk{n}{F}\times_{i^n(F)}I^n(F).
\]

%It follows, for example, that for all $n$, the kernel of $\mwk{n}{F}\to\milk{n}{F}$ is isomorphic to $I^{n+1}(F)$ and, 
%for $n\geq 0$,  the inclusion
%$I^{n+1}(F)\to\mwk{n}{F}$ is given by $\pfist{a_1,\ldots,a_{n+1}}\mapsto \eta[a_1]\cdots[a_n]$.
It follows that for all $n$ there is a natural short exact sequence 
\[
0\to I^{n+1}(F)\to \mwk{n}{F}\to \milk{n}{F}\to 0
\]
where the inclusion $I^{n+1}(F)\to\mwk{n}{F}$ is given by $\pfist{a_1,\ldots,a_{n+1}}\mapsto \eta[a_1]\cdots[a_n]$.

Similarly, for $n\geq 0$, there is a short exact sequence
\[
0\to 2\milk{n}{F}\to\mwk{n}{F}\to I^n(F)\to 0
\]
where the inclusion $2\milk{n}{F}\to\mwk{n}{F}$ is given (for $n\geq 1$) by 
$
2\{ a_1,\ldots,a_n\}\mapsto h[a_1]\cdots[a_n]
$.
Observe that, when $n\geq 2$,
\[
h[a_1][a_2]\cdots [a_n]=([a_1][a_2]-[a_2][a_1])[a_3]\cdots [a_n]=[a^2_1][a_2]\cdots[a_n].
\]
(The first equality follows from Lemma \ref{lem:mwkid1} (3) below, the second from the observation that 
$[a^2_1]\cdots [a_n]\in \ker{\mwk{n}{F}\to I^n(F)}=2\milk{n}{F}$ and the fact, which follows from Morel's theorem, that the composite
$2\milk{n}{F}\to\mwk{n}{F}\to\milk{n}{F}$ is the natural inclusion map.)

When $n=0$ we have an isomorphism of rings
\[
\gw{F}\cong W(F)\times_{\Z/2}\Z\cong \mwk{0}{F}.
\]
Under this isomorphism $\pfist{a}$ corresponds to $\eta [a]$ and $\wgen{a}$ corresponds to $\eta[a]+1$. (Observe that 
with this identification, $h=\eta[-1]+2=\wgen{1}+\wgen{-1}\in\mwk{0}{F}=\gw{F}$, as expected.) 

Thus each $\mwk{n}{F}$ has the structure of a $\gw{F}$-module (and hence also of a $\GR{\Z}{F^\times}$-module), with the action given by 
$\ffist{a}\cdot([a_1]\cdots [a_n])=\eta[a] [a_1]\cdots [a_n]$.

We record here some elementary identities in Milnor-Witt $K$-theory which we will need below.
\begin{lem}\label{lem:mwkid1} 
Let $a,b\in F^\times$.  The following identities hold in the Milnor-Witt $K$-theory of $F$:
\begin{enumerate}
\item $[a][-1]=[a][a]$.
\item $[ab]=[a]+\wgen{a}[b]$.
\item $[a][b]=-\wgen{-1}[b][a]$.
\end{enumerate} 
\end{lem}

\begin{proof}\ 

\begin{enumerate}
\item See, for example, the proof of Lemma 2.7 in \cite{hutchinson:tao}.
\item $\wgen{a}{b}=(\eta[a]+1)[b]=\eta[a][b]+[b]=[ab]-[a]$.
\item See \cite{hutchinson:tao}, Lemma 2.7.
\end{enumerate}
\end{proof}
\subsection{Homology of Groups}

Given a group $G$ and a $\GR{\Z}{G}$-module $M$, $\ho{n}{G}{M}$ will denote the $n$th homology group of $G$ with coefficients in the 
module $M$. $B_\bullet(G)$ will denote the \emph{right bar resolution of $G$}: $B_n(G)$ is the free right $\GR{\Z}{G}$-module with basis
the elements $[g_1|\cdots |g_n]$, $g_i\in G$. ($B_0(G)$ is isomorphic to $\GR{\Z}{G}$ with generator the symbol $[\ ]$.) The boundary 
$d=d_n:B_n(G)\to B_{n-1}(G)$, $n\geq 1$, is given by 
\[
d([g_1|\cdots|g_n])= \sum_{i=o}^{n-1}(-1)^i[g_1|\cdots|\hat{g_i}|\cdots|g_n]+(-1)^n[g_1|\cdots|g_{n-1}]\fgen{g_n}.
\]  
The augmentation $B_0(G)\to \Z$ makes $B_\bullet(G)$ into a free resolution of the trivial $\GR{\Z}{G}$-module $\Z$, and thus 
$\ho{n}{G}{M}=H_n(B_\bullet(G)\otimes_{\GR{\Z}{G}}M)$.

If $C_\bullet = (C_q,d)$ is a non-negative complex of $\GR{\Z}{G}$-modules, then 
$E_{\bullet,\bullet}:=B_\bullet(G)\otimes_{\gr{\Z}{G}}C_\bullet$ 
is a double complex  of abelian groups.
Each of the two filtrations on $E_{\bullet,\bullet}$ gives a spectral sequence  converging to the homology of the total complex of 
$E_{\bullet,\bullet}$,  which is by definition, 
$\ho{\bullet}{G}{C}$. (see, for example, Brown, \cite{brown:coh}, Chapter VII).

The first spectral sequence has the form
\[
E^2_{p,q}=\ho{p}{G}{H_q(C)}\Longrightarrow \ho{p+q}{G}{C}.
\] 
In the special case that there is a weak equivalence $C_\bullet \to \Z$ (the complex consisting of the trivial module 
$\Z$ concentrated in dimension $0$), it follows that $\ho{\bullet}{G}{C}=\ho{\bullet}{G}{\Z}$. 

The second spectral sequence has the form 
\[
E^1_{p,q}=\ho{p}{G}{C_q}\Longrightarrow\ho{p+q}{G}{C}.
\]
Thus, if $C_\bullet$ is weakly equivalent to $\Z$, this gives a spectral sequence converging to $\ho{\bullet}{G}{\Z}$. 

Our analysis of the homology of  special linear groups will exploit the action of these
groups on certain permutation modules. It is straightforward to compute the map induced on homology groups by  a map of 
permutation modules. We recall the following basic principles (see, for example, \cite{hutchinson:mat}): 
If $G$ is a group and if $X$ is a $G$-set, then Shapiro's Lemma says 
that
\[
\ho{p}{G}{\Z[X]}\cong \bigoplus_{y\in X/G}\ho{p}{G_y}{\Z},
\]
the isomorphism being induced by the maps 
\[
\ho{p}{G_y}{\Z}\to \ho{p}{G}{\Z[X]}
\]
described at the level of chains by
\[
B_p\otimes_{\gr{\Z}{G_y}}\Z\to B_p\otimes_{\gr{\Z}{G}}\Z[X],\quad z\otimes 1\mapsto z\otimes y.
\]

Let $X_i$, $i=1,2$ be transitive $G$-sets. Let $x_i\in X_i$ and let $H_i$ be the stabiliser of $x_i$, $i=1,2$. Let $\phi:\Z[X_1]\to\Z[X_2]$ 
be a map of $\Z[G]$-modules with
\[
\phi(x_1)=\sum_{g\in G/H_2}n_ggx_2,\qquad \mbox{ with } n_g\in\Z.
\] 
Then the induced map $\phi_\bullet:\ho{\bullet}{H_1}{\Z}\to\ho{\bullet}{H_2}{\Z}$ is given by the formula
\begin{eqnarray}\label{formula}
\phi_\bullet(z)
=\sum_{g\in H_1\backslash G/H_2}n_g\mathrm{cor}^{H_2}_{g^{-1}H_1g\cap H_2}
\mathrm{res}^{g^{-1}H_1g}_{g^{-1}H_1g\cap H_2}\left(g^{-1}\cdot z\right)
\end{eqnarray}

There is an obvious extension of this formula to non-transitive $G$-sets.

\subsection{Homology of $\specl{n}{F}$ and Milnor-Witt $K$-theory}\label{sec:hommwk}

Let $F$ be an infinite field.  

The theorem of Matsumoto and Moore (\cite{mat:pres}, \cite{moore:pres}) gives a presentation of the 
group $\hoz{2}{\specl{2}{F}}$. It has the following form: 
The generators are symbols $\an{a_1,a_1}$, $a_i\in F^\times$, subject to the relations:
\begin{enumerate}
\item[(i)] $\an{a_1,a_2}=0$ if $a_i=1$ for some $i$
\item[(ii)] $\an{a_1,a_2}=\an{a_2^{-1}, a_1}$
\item[(iii)] $\an{a_{1},a_2b_2} +\an{a_2,b_2}=\an{a_{1}a_2,b_2}+\an{a_1,a_2}$
\item[(iv)] $\an{a_1,a_2}=\an{a_1,-a_{1}a_2}$
\item[(v)] $\an{a_1,a_2}=\an{a_{1},(1-a_{1})a_2}$       
\end{enumerate}

It can be shown that for all $n\geq 2$, $\mwk{n}{F}$ admits a (generalised) Matsumoto-Moore presentation:

\begin{thm}[\cite{hutchinson:tao}, Theorem 2.5]\label{thm:matmoore}
For $n\geq 2$, $\mwk{n}{F}$ admits the following presentation as an additive group:

Generators: The elements $[a_1][a_2]\cdots[a_n]$, $a_i\in F^\times$.

Relations:\ \\
\begin{enumerate}
\item[(i)] $[a_1][a_2]\cdots[a_n]=0$ if $a_i=1$ for some $i$.
\item[(ii)] $[a_1]\cdots[a_{i-1}][a_i]\cdots[a_n]=[a_1]\cdots[a_i^{-1}][a_{i-1}]\cdots[a_n]$
\item[(iii)] $[a_1]\cdots[a_{n-1}][a_nb_n]+[a_1]\cdots\widehat{[a_{n-1}]}[a_{n}][b_n]=[a_1]\cdots[a_{n-1}a_n][b_n]+
[a_1]\cdots[a_{n-1}][a_n]$
\item[(iv)] $[a_1]\cdots[a_{n-1}][a_n]=[a_1]\cdots[a_{n-1}][-a_{n-1}a_n]$
\item [(v)]$[a_1]\cdots[a_{n-1}][a_n]=[a_1]\cdots[a_{n-1}][(1-a_{n-1})a_n]$
\end{enumerate}
\end{thm}

In particular, it follows when $n=2$ that there is a natural isomorphism $\mwk{2}{F}\cong \hoz{2}{\specl{2}{F}}$. This last 
fact is essentially 
due to Suslin (\cite{sus:tors}). A more recent proof, which we will need to invoke below,  has been given by Mazzoleni (\cite{mazz:sus}).

Recall that Suslin (\cite{sus:homgln}) has constructed a natural surjective homomorphism $\hoz{n}{\genl{n}{F}}\to\milk{n}{F}$ whose 
kernel is the image of $\hoz{n}{\genl{n-1}{F}}$.  

In \cite{hutchinson:tao2}, the authors proved that the map $\hoz{3}{\specl{3}{F}}\to \hoz{3}{\genl{3}{F}}$ is injective, that 
the image of the composite $\hoz{3}{\specl{3}{F}}\to \hoz{3}{\genl{3}{F}}\to\milk{3}{F}$ is $2\milk{3}{F}$ and that the kernel 
of this composite is precisely the image of $\hoz{3}{\specl{2}{F}}$. 

In the next section we will construct natural homomorphisms $\T{n}\circ\ee{n}:\hoz{n}{\specl{n}{F}}\to \mwk{n}{F}$, 
in a manner entirely analogous to Suslin's construction. In particular, the image of $\hoz{n}{\specl{n-1}{F}}$ is contained 
in the kernel of $\T{n}\circ\ee{n}$ and the diagrams 
\[
\xymatrix{
\hoz{n}{\specl{n}{F}}\ar[r]\ar[d]&\mwk{n}{F}\ar[d]\\
\hoz{n}{\genl{n}{F}}\ar[r]&\milk{n}{F}
}
\]
commute.  It follows that the image of $\T{3}\circ\ee{3}$ is $2\milk{3}{F}\subset\mwk{3}{F}$, and its kernel is the image of 
$\hoz{3}{\specl{2}{F}}$.

\section{The algebra $\SSn{F^\bullet}$}\label{sec:prelim}
In this section we introduce a graded algebra functorially associated to $F$ which admits a natural homomorphism to Milnor-Witt 
$K$-theory and from the homology of $\specl{n}{F}$. It is the analogue of Suslin's algebra $S_\bullet(F)$ in \cite{sus:tors}, which 
admits homomorphisms to Milnor $K$-theory and from the homology of $\genl{n}{F}$. However, we will need to modify this algebra in 
the later sections below, by projecting onto the `\mult' part, in order to derive our results about the homology of $\specl{n}{F}$.

We say that a finite set of vectors $v_1,\ldots,v_q$ in an $n$-dimensional vector space $V$ \emph{are in general position} if 
every subset of size $\mathrm{min}(q,n)$ is linearly independent. 

If $v_1,\ldots, v_q$ are elements of the $n$-dimensional vector space $V$ and if $\mathcal{E}$ is an ordered  basis of $V$, 
we let $[v_1|\cdots|v_q]_\mathcal{E}$   denote the $n\times q$ matrix whose $i$-th column is the components of $v_i$ with respect to the 
basis $\mathcal{E}$.

\subsection{Definitions}
For a field  $F$ and finite-dimensional vector spaces $V$ and $W$, we let $\xgen{p}{W,V}$ denote the set of all ordered $p$-tuples
of the form
\[
((w_1,v_1),\ldots,(w_p,v_p))
\]
where $(w_i,v_i)\in W\oplus V$ and the $v_i$ are in general position. 
We also define $\xgen{0}{W,V}:=\emptyset$. 
$\xgen{p}{W,V}$ is naturally an $\aff{W}{V}$-module, where 
\[
\aff{W}{V}:= 
{\begin{pmatrix}
\id{W}&\hom{}{V}{W}\\
0&\gnl{V}
\end{pmatrix}}
\subset \gnl{W\oplus V}
\]
Let $\cgen{p}{W,V}=\Z[\xgen{p}{W,V}]$, the free abelian group with basis the elements of $\xgen{p}{W,V}$. We obtain a complex, 
$\cgen{\bullet}{W,V}$, of $\aff{W}{V}$-modules by introducing the natural simplicial boundary map
\begin{eqnarray*}
d_{p+1}:\cgen{p+1}{W,V}&\to&\cgen{p}{W,V}\\
((w_1,v_1),\ldots,(w_{p+1},v_{p+1}))&\mapsto& \sum_{i=1}^{p+1}(-1)^{i+1}
((w_1,v_1),\ldots,\widehat{(w_i,v_i)},\ldots,(w_{p+1},v_{p+1}))
\end{eqnarray*}
\begin{lem} If $F$ is infinite, then $\homol{p}{\cgen{\bullet}{W,V}}=0$ for all $p$.
\end{lem}
\begin{proof}
If 
\[
z=\sum_in_i((w^i_1,v^i_1),\ldots,(w^i_p,v^i_p))\in \cgen{p}{W,V}
\]
 is a cycle, then since $F$ is infinite,
 it is possible to choose $v\in V$ such that $v,v^i_1,\ldots,v^i_p$ are in general position for all $i$. Then $z=d_{p+1}((-1)^ps_v(z))$ 
where $s_v$ is the `partial homotopy operator' defined by 
\[
s_v((w_1,v_1),\ldots,(w_p,v_p))=
\left\{
\begin{array}{ll}
((w_1,v_1),\ldots,(w_p,v_p),(0,v)), &\mbox{ if } v,v_1,\ldots v_p \mbox{ are in general position},\\
0,&\mbox{ otherwise}\\
\end{array}
\right.
\]
\end{proof}

We will assume our field $F$ is infinite for the remainder of this section. (In later sections, it will even be assumed to be of 
characteristic zero.)

If $n=\dim{F}{V}$, we let $H(W,V):=\ker{d_n}=\image{d_{n+1}}$. This is an $\aff{W}{V}$-submodule of $\cgen{n}{W,V}$. Let $\SSS(W,V):=
\ho{0}{\saff{W}{V}}{H(W,V)}=H(W,V)_{\saff{W}{V}}$ where $\saff{W}{V}:=\aff{W}{V}\cap\spcl{W\oplus V}$.  

If $W'\subset W$, there are natural inclusions $\xgen{p}{W',V}\to \xgen{p}{W,V}$ inducing a map of complexes of $\aff{W'}{V}$-modules 
$\cgen{\bullet}{W',V}\to\cgen{\bullet}{W,V}$.

When $W=0$, we will use the notation, $\xgen{p}{V}$, $\cgen{p}{V}$, $H(V)$ and $\SSS(V)$ instead of 
$\xgen{p}{0,V}$, $\cgen{p}{0,V}$, $H(0,V)$ and $\SSS(0,V)$ 
%If $V=F^n$, $W=F^m$ (resp. $W=0$), we use the notation $\comp{p}{m,n}{F}$, $H_{m,n}(F)$ and $\SSS_{m,n}(F)$ (resp.
% $\comp{p}{n}{F}$, $H_{n}(F)$ and  $\SSS_{n}(F)$) instead of $\gcomp{p}{W}{V}$, $H(W,V)$ and $\SSS(W,V)$.

Since, $\aff{W}{V}/\saff{W}{V}\cong F^\times$, any homology group of the form $\ho{i}{\saff{W}{V}}{M}$ where $M$ is a $\aff{W}{V}$-module 
is naturally a $\grf{F}$-module: 
If $a\in F^\times$ and if  $g\in \aff{W}{V}$ is any element of determinant $a$, then the action of $a$ is
the map on homology induced by conjugation by $g$ on $\aff{W}{V}$ and multiplication by $g$ on $M$. 
In particular, the groups $\SSS(W,V)$ are $\grf{F}$-modules.

Let $e_1,\ldots, e_n$ denote the standard basis of $F^n$. Given $a_1,\ldots,a_n\in F^\times$, we let 
$\ssb{a_1,\ldots,a_n}$ denote the class of $d_{n+1}(e_1,\ldots,e_n,a_1e_1+\cdots+a_ne_n)$ in $\SSS(F^n)$.   
If $b\in F^\times$, then 
$\fgen{b}\cdot\ssb{a_1,\ldots,a_n}$ is represented by 
\[
d_{n+1}(e_1,\ldots,be_i,\ldots,e_n,a_1e_1+\cdots a_ibe_i
\cdots +a_ne_n)
\]
 for any $i$. (As a lifting of $b\in F^\times$, choose the diagonal matrix with $b$ in the $(i,i)$-position and $1$ in all 
other diagonal positions.)

\begin{rem}\label{rem:ssn}
Given $x= (v_1,\ldots,v_v,v)\in\xgen{n+1}{F^n}$, let $A=[v_1|\cdots|v_n]\in\genl{n}{F}$ of determinant $\det{A}$ and 
let $A'=\mathrm{diag}(1,\ldots,1,\det{A})$. Then $B=A'A^{-1}\in\specl{n}{F}$ and thus $x$ is in the $\specl{n}{F}$-orbit 
of $(e_1,\ldots,e_{n-1},\det{A} e_n,A'w)$ with $w=A^{-1}v$, and hence $d_{n+1}(x)$ represents the element 
$\fgen{\det{A}}\ssb{w}$ in $\SSn{F^n}$.

\end{rem}
\begin{thm}\label{thm:pres}
$\SSS(F^n)$ has the following presentation as a $\gr{\Z}{F^\times}$-module:

Generators:  The elements $\ssb{a_1,\ldots,a_n}$, $a_i\in F^\times$

Relations: For all $a_1,\ldots,a_n\in F^\times$ and for all $b_1,\ldots,b_n\in F^\times$ with $b_i\not=b_j$ for $i\not= j$
\[
\ssb{b_1a_1,\ldots,b_na_n}-\ssb{a_1,\ldots,a_n}=\sum_{i=1}^{n}(-1)^{n+i}\fgen{(-1)^{n+i}a_i}
\ssb{a_1(b_1-b_i),\ldots,\widehat{a_i(b_i-b_i)},\ldots,a_n(b_n-b_i),b_i}.
\]
\end{thm}
\begin{proof}
Taking $\specl{n}{F}$-coinvariants of the exact sequence of $\gr{\Z}{\genl{n}{F}}$-modules
\[
\xymatrix{
\cgen{n+2}{F^n}\ar[r]^-{d_{n+2}}&\cgen{n+1}{F^n}\ar[r]^-{d_{n+1}}&H(F^n)\ar[r]&0
}
\]
gives the exact sequence of $\grf{F}$-modules
\[
\xymatrix{
\cgen{n+2}{F^n}_{\specl{n}{F}}\ar[r]^-{d_{n+2}}&\cgen{n+1}{F^n}_{\specl{n}{F}}\ar[r]^-{d_{n+1}}&\SSn{F^n}\ar[r]&0.
}
\]

It is straightforward to verify that 
\[
\xgen{n+1}{F^n}\cong \coprod_{a=(a_1,\ldots,a_n),a_i\not= 0}\genl{n}{F}\cdot(e_1,\ldots,e_n,a)
\]
as a $\genl{n}{F}$-set. It follows that 
\[
\cgen{n+1}{F^n}\cong \bigoplus _{a}\gr{\Z}{\genl{n}{F}}\cdot(e_1,\ldots,e_n,a)
\]
as a $\gr{\Z}{\genl{n}{F}}$-module, and thus that 
\[
\cgen{n+1}{F^n}_{\specl{n}{F}}\cong \bigoplus _{a}\grf{F}\cdot(e_1,\ldots,e_n,a)
\]
as a $\grf{F}$-module.

Similarly, every element of $\xgen{n+2}{F^n}$ is in the $\genl{n}{F}$-orbit of a unique element of the form 
$(e_1,\ldots,e_n,a,b\cdot a)$ where $a=(a_1,\ldots,a_n)$ with $a_i\not=0$ for all $i$ and $b=(b_1,\ldots,b_n)$ with $b_i\not= 0$ 
for all $i$ and $b_i\not= b_j$ for all $i\not=j$, and $b\cdot a:= (b_1a_1,\ldots,b_na_n)$. Thus 
\[
\xgen{n+2}{F^n}\cong \coprod_{(a,b)}\genl{n}{F}\cdot(e_1,\ldots,e_n,a,b\cdot a)
\]
as a $\genl{n}{F}$-set and 
\[
\cgen{n+2}{F^n}_{\specl{n}{F}}\cong \bigoplus _{(a,b)}\grf{F}\cdot(e_1,\ldots,e_n,a,b\cdot a)
\]
as a $\grf{F}$-module.

So $d_{n+1}$ induces an isomorphism
\[
\frac{\oplus\grf{F}\cdot(e_1,\ldots,e_n,a)}{\il{d_{n+2}(e_1,\ldots,e_n,a,b\cdot a)| (a,b)}}\cong\SSn{F^n}.
\]

Now $d_{n+2}(e_1,\ldots,e_n,a,b\cdot a)=$
\begin{eqnarray*}
\sum_{i=1}^n(-1)^{i+1}(e_1,\ldots,\hat{e_i},\ldots,e_n,a,b\cdot a)+(-1)^i\big((e_1,\ldots,e_n,b\cdot a)-(e_1,\ldots,e_n,a)\big).
\end{eqnarray*}

Applying the idea of Remark \ref{rem:ssn} to the terms $ (e_1,\ldots,\hat{e_i},\ldots,e_n,a,b\cdot a)$ in the sum above, we let 
$M_i(a):=[e_1|\cdots|\hat{e_i}|\cdots|e_n|a]$ and $\delta_i=\det{M_i(a)}=(-1)^{n-i}a_i$.  Since 
\[
M_i(a)^{-1}=
\begin{pmatrix}
1&\hdots&0&-a_1/a_i&0&\hdots&0\\
0&\ddots&\vdots&\vdots&\vdots&\vdots&\vdots\\
0&\hdots&1&-a_{i-1}/a_i&0&\hdots&0\\
0&\hdots&0&-a_{i+1}{a_i}&1&\hdots&0\\
0&\hdots&0&\vdots&0&\ddots&0\\
0&\hdots&0&-a_n/a_i&0&\hdots&1\\
0&\hdots&0&1/a_i&0&\hdots&0
\end{pmatrix}
\]
it follows that $d_{n+1}(e_1,\ldots,\hat{e_i},\ldots,e_n,a,b\cdot a)$ represents $\fgen{\delta_i}\ssb{w_i}\in \SSn{F^n}$ where 
$w_i=M_i(a)^{-1}(b\cdot a)=(a_1(b_1-b_i),\ldots,\widehat{a_i(b_i-b_i)},\ldots,a_n(b_n-b_i),b_i)$. This proves the theorem.
\end{proof}

\subsection{Products}

If $W'\subset W$, there is a natural bilinear pairing 
\[
\cgen{p}{W',V}\times \cgen{q}{W}\to \cgen{p+q}{W\oplus V},\qquad (x,y)\mapsto x\ssprod  y
\]
 defined on the basis elements by
\[
((w'_1,v_1),\ldots,(w'_p,v_p))\ssprod (w_1,\ldots,w_q):= 
\left((w'_1,v_1),\ldots,(w'_p,v_p),(w_1,0),\ldots,(w_q,0)\right).
\]

This pairing satisfies $d_{p+q}(x\ssprod y)=d_p(x)\ssprod y + (-1)^px\ssprod d_q(y)$.

Furthermore, if $\alpha\in \aff{W'}{V}\subset\gnl{W\oplus V}$ then $(\alpha x)\ssprod y=\alpha(x\ssprod y)$, and if 
$\alpha\in \gnl{V}\subset\aff{W'}{V}\subset\gnl{W\oplus V}$ and $\beta\in \gnl{W}\subset \gnl{W\oplus V}$, then 
$(\alpha x)\ssprod (\beta y)=(\alpha\cdot\beta)(x\ssprod y)$. (However, if $W'\not= 0$ then the images of $\aff{W'}{V}$ and $\gnl{W}$ 
in $\gnl{W\oplus V}$ don't commute.)

In particular, there are induced pairings on homology groups 
\[ 
H(W',V)\otimes H(W)\to H(W\oplus V),
\]
  which in turn induce well-defined 
pairings 
\[
\SSS(W',V)\otimes H(W)\to \SSS(W,V) \mbox{ and }\SSS(V)\otimes \SSS(W)\to \SSS(W\oplus V).
\]
 Observe further that this latter pairing is 
$\gr{\Z}{F^\times}$-balanced:
 If $a\in F^\times$, $x\in \SSS(W)$ and $y\in \SSS(V)$, then $(\fgen{a}x)\ssprod y = x\ssprod (\fgen{a}y)=\fgen{a}(x\ssprod y)$. 
Thus there is a 
well-defined map 
\[
\SSS(V)\otimes_{\gr{\Z}{F^\times}}\SSS(W)\to \SSS(W\oplus V).
\]

In particular, the groups $\{ H(F^n)\}_{n\geq 0}$ form a natural graded (associative) algebra, 
and the groups $\{ \SSS(F^n)\}_{n\geq 0}= \SSn{F^\bullet}$ form a graded associative $\gr{\Z}{F^\times}$-algebra.

The following explicit formula for the product in $\SSn{F^\bullet}$  will be needed below:
\begin{lem}\label{lem:sprod}
Let $a_1,\ldots,a_n$ and $a'_1,\ldots,a'_m$ be elements of $F^\times$.  Let $b_1,\ldots,b_n,b'_1,\ldots,b'_m$ be any elements of 
$F^\times$ satisfying $b_i\not=b_j$ for $i\not= j$ and $b'_s\not=b'_t$ for $s\not= t$.  

Then $\ssb{a_1,\ldots,a_n}\ssprod\ssb{a'_1,\ldots,a'_m}=$
\begin{tiny}
\begin{eqnarray*}
\sum_{i=1}^n\sum_{j=1}^m(-1)^{m+n+i+j}\fgen{(-1)^{i+j}a_ia'_j}
\ssb{a_1(b_1-b_i),\ldots,\widehat{a_i(b_i-b_i)},\ldots,b_i,a'_1(b'_1-b'_j),\ldots,\widehat{a'_j(b'_j-b'_j)},\ldots,b'_j}\\
+(-1)^n\sum_{i=1}^n(-1)^{i+1}\fgen{(-1)^{i+1}a_i}\ssb{a_1(b_1-b_i),\ldots,\widehat{a_i(b_i-b_i)},\ldots,b_i,b'_1a'_1,\ldots,b'_ma'_m}\\
+(-1)^m\sum_{j=1}^m(-1)^{j+1}\fgen{(-1)^{j+1}a'_j}\ssb{b_1a_1,\ldots,b_na_n,a'_1(b'_1-b'_j),
\ldots,\widehat{a'_j(b'_j-b'_j)},\ldots,b'_j}\\
+\ssb{b_1a_1,\ldots,b_na_n,b'_1a'_1,\ldots,b'_ma'_m} 
\end{eqnarray*}
\end{tiny}
\end{lem}
\begin{proof}
This is an entirely straightforward calculation using the definition of the product, Remark \ref{rem:ssn}, the matrices 
$M_i(a)$, $M_j(a')$ as in the proof of Theorem \ref{thm:pres}, and the partial homotopy operators $s_v$ with 
$v=(a_1b_1,\ldots,a_nb_n,a'_1b'_1,\ldots,a'_mb'_m)$.
\end{proof}
\subsection{The maps $\ee{V}$}

If $\dim{F}{V}=n$, then the exact sequence of $\gnl{V}$-modules 
\begin{eqnarray*}\label{seq:hv}
\xymatrix{
0\ar[r] &H(V)\ar[r]&\cgen{n}{V}\ar[r]^-{d_n}&\cgen{n-1}{V}\ar[r]^-{d_{n-1}}&\cdots\ar[r]^-{d_1}&\cgen{0}{V}=\Z\ar[r]&0
}
\end{eqnarray*}
gives rise to an iterated connecting homomorphism
\[
\ee{V}:\hoz{n}{\spcl{V}}\to\ho{0}{\spcl{V}}{H(V)}=\SSS(V).
\]

Note that the collection of groups $\{\hoz{n}{\specl{n}{F}}\}$ form a graded $\gr{\Z}{F^\times}$-algebra under the 
graded product induced by exterior product on homology, together with the obvious direct sum homomorphism $\specl{n}{F}\times\specl{m}{F}\to 
\specl{n+m}{F}$.

\begin{lem}\label{lem:ee}
The maps $\ee{n}:\hoz{n}{\specl{n}{F}}\to \SSS(F^n)$, $n\geq 0$, 
give a well-defined homomorphism of graded $\gr{\Z}{F^\times}$-algebras; i.e.
\begin{enumerate}
\item If $a\in F^\times$ and $z\in \hoz{n}{\specl{n}{F}}$, then $\ee{n}(\fgen{a}z)=\fgen{a}\ee{n}(z)$ in $\SSS(F^n)$, and
\item If $z\in \hoz{n}{\specl{n}{F}}$ and $w\in \hoz{m}{\specl{m}{F}}$ then 
\[
\ee{n+m}(z\times w)=\ee{n}(z)\ssprod \ee{m}(w)\mbox{ in } \SSS(F^{n+m}).
\]
\end{enumerate}
\end{lem}
\begin{proof}\ 

\begin{enumerate}
\item The exact sequence above is a sequence of $\gnl{V}$-modules and hence all of the connecting homomorphisms 
$\delta_i:\ho{n-i+1}{\spcl{V}}{\image{d_i}}\to \ho{n-i}{\spcl{V}}{\ker{d_i}}$ are $F^\times$-equivariant.
\item Let $\compt{\bullet}{V}$ denote the truncated complex.
\[
\compt{p}{V}=
\left\{
\begin{array}{cc}
\cgen{p}{V},& p\leq \dim{F}{V}\\
0,& p>\dim{F}{V}
\end{array}
\right.
\]  
Thus $H(V)\to\compt{\bullet}{V}$ is a weak equivalence of complexes (where we regard $H(V)$ as a complex concentrated in dimension 
$\dim{}{V}$). Since the complexes $\compt{\bullet}{V}$ are complexes of free abelian groups, it follows that for two vector spaces 
$V$ and $W$, the map $H(V)\otimes_{\Z}H(W)\to T_\bullet(V,W)$ is an equivalence of complexes, where $T_\bullet(V,W)$ 
is the total complex of the 
double complex $\compt{\bullet}{V}\otimes_{\Z}\compt{\bullet}{W}$. Now $T_\bullet(V,W)$ is a complex of $\spcl{V}\times\spcl{W}$-modules, 
and the product $\ssprod$ induces a commutative diagram of complexes of $\spcl{V}\times\spcl{W}$-complexes:
\[
\xymatrix{
H(V)\otimes_{\Z}H(W)\ar[r]\ar[d]^-{\ssprod}&\compt{\bullet}{V}\otimes\compt{\bullet}{W}\ar[d]^-{\ssprod}\\
H(V\oplus W)\ar[r]&\compt{\bullet}{V\oplus W}
}
\]
which, in turn, induces a commutative diagram
\[
\xymatrix{
\hoz{n}{\spcl{V}}\otimes\hoz{m}{\spcl{W}}\ar[r]^-{\ee{V}\otimes\ee{W}}\ar[d]^-{\times}
&
\ho{0}{\spcl{V}}{H(V)}\otimes\ho{0}{\spcl{W}}{H(W)}\ar[d]^-{\times}\\
\ho{n+m}{\spcl{V}\times\spcl{W}}{\Z\otimes\Z}\ar[r]^-{\ee{T_\bullet}}\ar[d]
&
\ho{0}{\spcl{V}\times\spcl{W}}{H(V)\otimes H(W)}\ar[d]\\
\hoz{n+m}{\spcl{V\oplus W}}\ar[r]^-{\ee{V\oplus W}}
&
\ho{0}{\spcl{V\oplus W}}{H(V\oplus W)}
}
\]
(where $n=\dim{}{V}$ and $m=\dim{}{W}$).
\end{enumerate}
\end{proof}

\begin{lem}
If $V=W\oplus W'$ with $W'\not= 0$, then the composite 
\[
\xymatrix{
\hoz{n}{\spcl{W}}\ar[r]&\hoz{n}{\spcl{V}}\ar[r]^-{\ee{V}}&\SSS(V)
}
\]
is zero.
\end{lem}
\begin{proof}
The exact sequence of $\spcl{V}$-modules 
\[
0\to\ker{d_1}\to\cgen{1}{V}\to\Z\to 0
\] 
is split as a sequence of $\spcl{W}$-modules via the map $\Z\to\cgen{1}{V}, m\mapsto m\cdot e$ where $e$ is any nonzero element of $W'$.
It follows that the connecting homomorphism $\delta_1:\hoz{n}{\spcl{W}}\to\ho{n-1}{\spcl{W}}{\ker{d_1}}$ is zero.
\end{proof}

Let $\SH{n}{F}$ denote the cokernel of the map $\hoz{n}{\specl{n-1}{F}}\to\hoz{n}{\specl{n}{F}}$. 
It follows that the maps $\ee{n}$ give well-defined homomorphisms
$\SH{n}{F}\to\SSS(F^n)$, which yield a homomorphism of graded $\gr{\Z}{F^\times}$-algebras $\ee{\bullet}:\SH{\bullet}{F}\to\SSS(F^\bullet)$.

\subsection{The maps $\Dd{V}$}

Suppose now that $W$ and $V$ are vector spaces and that $\dim{}{V}=n$. Fix a basis $\mathcal{E}$ of $V$. The group $\aff{W}{V}$ acts 
transitively on $\xgen{n}{W,V}$ (with trivial stabilizers), 
while the orbits of $\saff{W}{V}$ are in one-to-one correspondence with the points of $F^\times$ via 
the correspondence 
\[
\xgen{n}{W,V}\to F^\times, \quad ((w_1,v_1),\ldots,(w_n,v_n))\mapsto \det\left([v_1|\cdots|v_n]_{\mathcal{E}}
\right).
\]
 Thus we have an induced isomorphism
\[
\xymatrix{
\ho{0}{\saff{W}{V}}{\cgen{n}{W,V}}\ar[r]^-{\det}_-{\cong}&\gr{\Z}{F^\times}.
}
\]
Taking $\saff{W}{V}$-coinvariants of the inclusion $H(W,V)\to\cgen{n}{W,V}$ then yields a homomorphism of $\gr{\Z}{F^\times}$-modules
\[
\Dd{W,V}:\SSS(W,V)\to \gr{\Z}{F^\times}.
\]

In particular, for each $n\geq 1$ we have a homomorphism of $\gr{\Z}{F^\times}$-modules $\Dd{n}:\SSS(F^n)\to \gr{\Z}{F^\times}$.

We will also set $\Dd{0}:\SSS(F^0)=\Z\to \Z$ equal to the identity map. Here $\Z$ is a trivial $F^\times$-module.

We set 
\[
\Aa{n}=
\left\{
\begin{array}{ll}
\Z, & n=0\\
\aug{F^\times},& n \mbox{ odd }\\
\gr{\Z}{F^\times},& n > 0  \mbox{ even}\\
\end{array}
\right.
\]

We have $\Aa{n}\subset \gr{\Z}{F^\times}$ for all $n$ and we make $\Aa{\bullet}$ into a graded algebra by using the multiplication on 
$\gr{\Z}{F^\times}$.

\begin{lem}\label{lem:dn}\ 

 \begin{enumerate}
\item The image of $\Dd{n}$ is $\Aa{n}$. 
\item The maps $\Dd{\bullet}:\SSS(F^\bullet)\to \Aa{\bullet}$ define a homomorphism of graded $\gr{\Z}{F^\times}$-algebras.
\item For each $n\geq 0$, the surjective map $\Dd{n}:\SSS(F^n)\to \Aa{n}$ has a $\gr{\Z}{F^\times}$-splitting. 
\end{enumerate}
\end{lem}

\begin{proof}\ 

\begin{enumerate}
\item Consider a generator $\ssb{a_1,\ldots,a_n}$ of $\SSS(F^n)$.

Let $e_1,\ldots, e_n$ be the standard basis of $F^n$. Let $a:=a_1e_1+\cdots+a_ne_n$. Then
\begin{eqnarray*}
\ssb{a_1,\ldots,a_n}&=&d_{n+1}(e_1,\ldots,e_n,a)\\
&=&\sum_{i=1}^n(-1)^{i+1}(e_1,\ldots,\widehat{e_i},\ldots,e_n,a)+(-1)^n(e_1,\ldots,e_n).
\end{eqnarray*}
Thus
\begin{eqnarray*}
\Dd{n}(\ssb{a_1,\ldots,a_n})&=&\sum_{i=1}^n(-1)^{i+1}\fgen{\det\left([e_1|\cdots|\widehat{e_i}|\cdots|e_n|a]\right)}+(-1)^n\fgen{1}\\
&=& \left\{
\begin{array}{ll}
\fgen{a_1}-\fgen{-a_2}+\cdots +\fgen{a_n}-\fgen{1}, & n \mbox{ odd}\\
\fgen{-a_1}-\fgen{a_2}+\cdots -\fgen{a_n}+\fgen{1},& n> 0 \mbox{ even}\\
\end{array}
\right.
\end{eqnarray*}

Thus, when $n$ is even, $\Dd{n}(\ssb{-1,1,-1,\ldots,-1,1})=\fgen{1}$ and $\Dd{n}$ maps onto $\gr{\Z}{F^\times}$.

When $n$ is odd, clearly, $\Dd{n}(\ssb{a_1,\ldots,a_n})\in \aug{F^\times}$. However, for any $a\in F^\times$, 
$\Dd{n}(\ssb{a,-1,1,\ldots,-1,1})=\ffist{a}\in\Aa{n}=\aug{F^\times}$.

\item Note that $\cgen{n}{F^n}\cong\gr{\Z}{\genl{n}{F}}$ naturally. Let $\mu$ be the homomorphism of additive groups
\begin{eqnarray*}
\mu:\gr{\Z}{\genl{n}{F}}\otimes\gr{\Z}{\genl{m}{F}}&\to&\gr{\Z}{\genl{n+m}{F}},\\
A\otimes B&\mapsto& 
\begin{pmatrix}A&0\\
0&B
\end{pmatrix}
\end{eqnarray*}
The formula $\Dd{m+n}(x\ssprod y)=\Dd{n}(x)\cdot\Dd{m}(y)$ now follows from the commutative diagram
\[
\xymatrix{
H(F^n)\otimes H(F^m)\ar[r]^-{\ssprod}\ar[d]&H(F^{n+m})\ar[d]\\
\cgen{n}{F^n}\otimes\cgen{m}{F^m}\ar[r]^-{\ssprod}\ar[d]^-{\cong}&\cgen{n+m}{F^{n+m}}\ar[d]^-{\cong}\\
\gr{\Z}{\genl{n}{F}}\otimes\gr{\Z}{\genl{m}{F}}\ar[r]^-{\mu}\ar[d]^-{\det\otimes\det}&\gr{\Z}{\genl{n+m}{F}}\ar[d]^-{\det}\\
\gr{\Z}{F^\times}\otimes\gr{\Z}{F^\times}\ar[r]^{\cdot}&\gr{\Z}{F^\times}\\
}
\]

\item When $n$ is even the maps $\Dd{n}$ are split surjections, since the image is a free module of rank $1$.

It is easy to verify that the map $\Dd{1}:\SSS(F)\to \Aa{1}=\aug{F^\times}$ is an isomorphism. Now let $E\in \SSS(F^2)$ be any 
element satisfying $\Dd{2}(E)=\fgen{1}$ (eg. we can take $E=\ssb{-1,1}$). Then for $n=2m+1$ odd, the composite 
$\SSS(F)\ssprod E^{\ssprod m} \to \SSS(F^n)\to \aug{F^\times}=\Aa{n}$ is an isomorphism.
\end{enumerate}
\end{proof}

We will let $\SSp{W,V}=\ker{\Dd{W,V}}$. Thus $\SSn{F^n}\cong \SSp{F^n}\oplus \Aa{n}$ as a $\gr{\Z}{F^\times}$-module by the results above.

Observe that it follows directly from the definitions that the image of $\ee{V}$ is contained 
in $\SSp{V}$ for any vector space $V$.

\subsection{The maps $\T{n}$}

\begin{lem}\label{lem:mwkid2}
If $n\geq 2$ and $b_1,\ldots,b_n$ are distinct elements of $F^\times$ then
\[
[b_1][b_2]\cdots[b_n]=\sum_{i=1}^n[b_1-b_i]\cdots[b_{i-1}-b_i][b_i][b_{i+1}-b_i]\cdots[b_n-b_i] \mbox{ in }\mwk{n}{F}.
\]
\end{lem}
\begin{proof}
We will use induction on $n$ starting with $n=2$: Suppose that $b_1\not= b_2\in F^\times$. 
Then 
\begin{eqnarray*}
[b_1-b_2]([b_1]-[b_2])&=& \left([b_1]+\wgen{b_1}\left[ 1-\frac{b_2}{b_1}\right]\right)\left(-\wgen{b_1}\left[\frac{b_2}{b_1}\right]\right)
\mbox{ by Lemma \ref{lem:mwkid1} (2)}\\
&=& -\wgen{b_1}[b_1]\left[\frac{b_2}{b_1}\right]\mbox{ since } [x][1-x]=0\\
&=& [b_1]([b_1]-[b_2])\mbox{ by Lemma \ref{lem:mwkid1}(2) again}\\
&=& [b_1]([-1]-[b_2])\mbox{ by Lemma \ref{lem:mwkid1} (1)}\\
&=& [b_1](-\wgen{-1}[-b_2])
= [-b_2][b_1]\mbox{ by Lemma \ref{lem:mwkid1} (3)}.\\
\end{eqnarray*}
Thus
\begin{eqnarray*}
[b_1][b_2-b_1]+[b_1-b_2][b_2]&=& -\wgen{-1}[b_2-b_1][b_1]+[b_1-b_2][b_2]\\
&=& -([b_1-b_2]-[-1])[b_1]+[b_1-b_2][b_2]\\
&=&-[b_1-b_2]([b_1]-[b_2])+[-1][b_1]\\
&=& -[-b_2][b_1]+[-1][b_1]=([-1]-[-b_2])[b_1]\\
&=& -\wgen{-1}[b_2][b_1]=[b_1][b_2]
\end{eqnarray*}
proving the case $n=2$. 

Now suppose that $n>2$ and that the result holds for $n-1$. Let $b_1,\ldots,b_n$ be distinct elements of $F^\times$.We wish to prove that 
\[
\bigg(\sum_{i=1}^{n-1}[b_1-b_i]\cdots[b_i]\cdots[b_{n-1}-b_i]\bigg)[b_n]=\sum_{i=1}^n[b_1-b_i]\cdots[b_i]\cdots[b_n-b_i].
\]
We re-write this as:
\[
\sum_{i=1}^{n-1}[b_1-b_i]\cdots[b_i]\cdots[b_{n-1}-b_i]([b_n]-[b_n-b_i])=[b_1-b_n]\cdots[b_{n-1}-b_n][b_n].
\]
Now
\begin{tiny}
\begin{eqnarray*}
[b_1-b_i]\cdots[b_i]\cdots[b_{n-1}-b_i]([b_n]-[b_n-b_i])&=&(-\wgen{-1})^{n-i}[b_1-b_i]\cdots[b_{n-1}-b_i]\bigg([b_i]([b_n]-[b_n-b_i])
\bigg)\\
&=& (-\wgen{-1})^{n-i}[b_1-b_i]\cdots[b_{n-1}-b_i]\bigg([b_i-b_n][b_n]\bigg)\\
&=&[b_1-b_i]\cdots [b_i-b_n]\cdots[b_{n-1}-b_i][b_n].
\end{eqnarray*}
\end{tiny}
So the identity to be proved reduces to
\[
\bigg(\sum_{i=1}^{n-1}[b_1-b_i]\cdots [b_i-b_n]\cdots[b_{n-1}-b_i]\bigg)[b_n]=[b_1-b_n]\cdots[b_{n-1}-b_n][b_n].
\]
Letting $b'_i=b_i-b_n$ for $1\leq i\leq n-1$, then $b_j-b_i=b'_j-b'_i$ for $i,j\leq n-1$ and this reduces to the case $n-1$.
\end{proof}

\begin{thm}\ 

\begin{enumerate}
\item For all $n\geq 1$, there is a well-defined homomorphism of $\gr{\Z}{F^\times}$-modules
\[
\T{n}:\SSS(F^n)\to \mwk{n}{F}
\]
sending $\ssb{a_1,\ldots,a_n}$ to $[a_1]\cdots[a_n]$. 

\item The maps $\{ \T{n}\}$ define a homomorphism of graded $\gr{\Z}{F^\times}$-algebras $\SSS(F^\bullet)\to \mwk{\bullet}{F}$:
We have
\[
\T{n+m}(x\ssprod y)=\T{n}(x)\cdot \T{m}(y),\quad \mbox{ for all } x\in \SSS(F^n), y\in \SSS(F^m).
\]
\end{enumerate}
\end{thm}

\begin{proof}\ 

\begin{enumerate} 
\item By Theorem \ref{thm:pres}, in order to show that $\T{n}$ is well-defined we must prove the identity
\[
[b_1a_1]\cdots[b_na_n]-[a_1]\cdots[a_n]=
\sum_{i=1}^n(-\wgen{-1})^{n+i}\wgen{a_i}[a_1(b_1-b_i)]\cdots[\widehat{a_i(b_i-b_i)}]\cdots[a_n(b_n-b_i][b_i]
\]  
in $\mwk{n}{F}$.

Writing $[b_ia_i]=[a_i]+\wgen{a_i}[b_i]$ and $[a_j(b_j-b_i)]=[a_j]+\wgen{a_j}[b_j-b_i]$ and expanding the products on 
both sides and using (3) of Lemma \ref{lem:mwkid1} to permute terms, this identity can be rewritten as 
\begin{eqnarray*}
\sum_{\emptyset\not=I\subset\{1,\ldots,n\}}(-\wgen{-1})^{\sgn{\sigma_I}}\wgen{a_{i_1}\cdots a_{i_k}}[a_{j_1}]\cdots[a_{j_s}]
[b_{i_1}]\cdots[b_{i_k}] = \\
\sum_{\emptyset\not=I\subset\{1,\ldots,n\}}(-\wgen{-1})^{\sgn{\sigma_I}}\wgen{a_{i_1}\cdots a_{i_k}}[a_{j_1}]\cdots[a_{j_s}]
\bigg(
\sum_{t=1}^{k}[b_{i_1}-b_{i_t}]\cdots [b_{i_t}]\cdots[b_{i_k}-b_{i_t}]
\bigg)
\end{eqnarray*}
where $I=\{ i_1<\cdots < i_k\}$ and the complement of $I$ is $\{ j_1<\cdots<j_s\}$ (so that $k+s=n$) and $\sigma_I$ is the permutation
\[
\begin{pmatrix}
1&\hdots&s&s+1&\hdots&n\\
j_1&\hdots&j_s&i_1&\hdots&i_k
\end{pmatrix}.
\] 

The result now follows from the identity 
of Lemma \ref{lem:mwkid2}.

\item We can assume that $x=\ssb{a_1,\ldots,a_n}$ and $y=\ssb{a'_1,\ldots,a'_m}$ with $a_i,a'_j\in F^\times$.
From the definition of $\T{n+m}$ and the formula of Lemma \ref{lem:sprod},\\
 $\T{n+m}(x\ssprod y)=$
\begin{tiny}
\begin{eqnarray*}
\sum_{i=1}^n\sum_{j=1}^m(-1)^{n+m+i+j}\fgen{(-1)^{i+j}a_ia'_j}
[a_1(b_1-b_i)]\cdots[\widehat{a_i(b_i-b_i)}]\cdots[b_i][a'_1(b'_1-b'_j)]\cdots[\widehat{a'_j(b'_j-b'_j)}]\cdots[b'_j]\\
+(-1)^n\sum_{i=1}^n(-1)^{i+1}\fgen{(-1)^{i+1}a_i}[a_1(b_1-b_i)]\cdots[\widehat{a_i(b_i-b_i)}]\cdots[b_i][b'_1a'_1]\cdots[b'_ma'_m]\\
+(-1)^m\sum_{j=1}^m(-1)^{j+1}\fgen{(-1)^{j+1}a'_j}[b_1a_1]\cdots[b_na_n][a'_1(b'_1-b'_j)]\cdots[\widehat{a'_j(b'_j-b'_j)}]\cdots[b'_j]\\
+[b_1a_1]\cdots[b_na_n][b_i][b'_1a'_1]\cdots[b'_ma'_m]
\end{eqnarray*}
\end{tiny}
which factors as $X\cdot Y$ with  $X=$
\begin{eqnarray*}
\sum_{i=1}^n(-1)^{n+i+1}\fgen{(-1)^{i+1}a_i}[a_1(b_1-b_i)]\cdots[\widehat{a_i(b_i-b_i)}]\cdots[b_i]+[b_1a_1]\cdots[b_na_n]\\
= [a_1]\cdots[a_n]=\T{n}(x) \mbox{ by part (1)}
\end{eqnarray*}
and $Y=$
\begin{eqnarray*}
\sum_{j=1}^m(-1)^{m+j+1}\fgen{(-1)^{j+1}a'_j}[a'_1(b'_1-b'_j)]\cdots[\widehat{a'_j(b'_j-b'_j)}]\cdots[b'_j]+[b'_1a'_1]
\cdots[b'_ma'_m]\\
= [a'_1]\cdots[a'_m]=\T{m}(y) \mbox{ by (1) again}.
\end{eqnarray*}
\end{enumerate}
\end{proof}

Note that 
$\T{1}$ is the natural surjective map $\SSn{F}\cong \aug{F^\times}\to\mwk{1}{F}$, $\ssb{a}\leftrightarrow \ffist{a}\mapsto [a]$. It has a 
nontrivial kernel in general.

Note furthermore that $\SH{2}{F}=\hoz{2}{\specl{2}{F}}$. It is well-known (\cite{sus:tors},\cite{mazz:sus}, and \cite{hutchinson:tao}) 
that $\hoz{2}{\specl{2}{F}}\cong \milk{2}{F}\times_{\milkt{2}{F}}I^2(F)\cong \mwk{2}{F}$.

In fact we have:

\begin{thm}\label{thm:T2} 
The composite $\T{2}\circ\ee{2}:\hoz{2}{\specl{2}{F}}\to \mwk{2}{F}$ is an isomorphism.
\end{thm}
\begin{proof}
For $p\geq 1$, let $\xgenp{p}{F}$ denote the set of all $p$-tuples $(x_1,\ldots,x_p)$ of points of $\proj{1}{F}$ and 
let $\xgenp{0}{F}=\emptyset$. We let $\cgenp{p}{F}$ denote the $\genl{2}{F}$ permutation module $\Z[\xgenp{p}{F}]$ and form a complex 
$\cgenp{\bullet}{F}$ using the natural simplicial boundary maps, $\bar{d}_p$. 
This complex is acyclic and the map $F^2\setminus\{0\}\to\proj{1}{F}$,
$v\mapsto\p{v}$ induces a map of complexes $\cgen{\bullet}{F^2}\to\cgenp{\bullet}{F}$.

Let $\hgenp{2}{F}:=\ker{\bar{d}_2:\cgenp{2}{F}\to\cgenp{1}{F}}$ and let $\sgenp{2}{F}=\ho{0}{\specl{2}{F}}{\hgenp{2}{F}}$.

We obtain a commutative diagram of $\specl{2}{F}$-modules with exact rows:
\[
\xymatrix{
\cgen{4}{F^2}\ar[d]\ar[r]^-{d_4}&\cgen{3}{F^2}\ar[d]\ar[r]^-{d_3}&H(F^2)\ar[d]\ar[r]&0\\
\cgenp{4}{F}\ar[r]^-{\bar{d}_4}&\cgenp{3}{F}\ar[r]^-{\bar{d}_3}&\hgenp{2}{F}\ar[r]&0\\
}
\]

Taking $\specl{2}{F}$-coinvariants gives the diagram
\[
\xymatrix{
\ho{0}{\specl{2}{F}}{\cgen{4}{F^2}}\ar[d]\ar[r]^-{d_4}&\ho{0}{\specl{2}{F}}{\cgen{3}{F^2}}
\ar[d]\ar[r]^-{d_3}&\SSn{F^2}\ar[d]^-{\phi}\ar[r]&0\\
\ho{0}{\specl{2}{F}}{\cgenp{4}{F}}\ar[r]^-{\bar{d}_4}&\ho{0}{\specl{2}{F}}{\cgenp{3}{F}}\ar[r]^-{\bar{d}_3}&\sgenp{2}{F}\ar[r]&0\\
}
\]

Now the  calculations of Mazzoleni, \cite{mazz:sus}, 
show that $\ho{0}{\specl{2}{F}}{\cgenp{3}{F}}\cong \gr{\Z}{F^\times/(F^\times)^2}$ via 
\[
\mbox{class of }(\infty,0,a)\mapsto \fgen{a}\in \gr{\Z}{F^\times/(F^\times)^2},
\]
where $a\in\proj{1}{F}=\p{e_1+ae_2}$ and $\infty:=\p{e_1}$. Furthermore $\sgenp{2}{F}\cong \gw{F}$ in such a way that the induced map 
$\gr{\Z}{F^\times/(F^\times)^2}\to\gw{F}$ is the natural one. 

Since $\ssb{a,b}=d_3(e_1,e_2,ae_1+be_2)$, it follows that $\phi(\ssb{a,b})=\wgen{a/b}=\wgen{ab}$ in $\gw{F}$.   

Associated to the complex $\cgenp{\bullet}{F}$ we have an iterated connecting homomorphism $\omega:\hoz{2}{\specl{2}{F}}\to 
\sgenp{2}{F}=\gw{F}$.  Observe that $\omega=\phi\circ\ee{2}$. 
In fact, (Mazzoleni, \cite{mazz:sus}, Lemma 5) the image of $\omega$ is $I^2(F)\subset \gw{F}$.

On the other hand, the module $\SSp{F^2}$ is generated by the elements\\
 $\ssbp{a,b}:=\ssb{a,b}-\Dd{2}(\ssb{a,b})\cdot E$ (where $E$, 
as above, denotes the element $\ssb{-1,1}$). 

Note that $\T{2}(\ssbp{a,b})=\T{2}(\ssb{a,b})=[a][b]$ since $\T{2}(E)=[-1][1]=0$ in $\mwk{2}{F}$.

Furthermore, 
\begin{eqnarray*}
\phi(\ssbp{a,b})&=&\phi(\ssb{a,b})-\Dd{2}(\ssb{a,b})\phi(E)\\
                &=& \wgen{ab}-(\fgen{-a}-\fgen{b}+\fgen{1})\wgen{-1}\\
                &=& \wgen{ab}-\wgen{a}+\wgen{-b}-\wgen{-1}\\
                &=&  \wgen{ab}-\wgen{a}-\wgen{b}+\wgen{1}\\
                &=& \pfist{a,b}
\end{eqnarray*}
(using the identity $\wgen{b}+\wgen{-b}=\wgen{1}+\wgen{-1}$ in $\gw{F}$).

Using these calculations  we thus obtain the commutative diagram
\[
\xymatrix{\hoz{2}{\specl{2}{F}}\ar[r]^-{\ee{2}}\ar[dr]^-{\omega}&\SSp{F^2}\ar[d]^-{\phi}\ar[r]^-{\T{2}}&\mwk{2}{F}\ar[dl]\\
&I^2(F)&\\
}
\] 
Now, the natural embedding $F^\times\to\specl{2}{F}$, $a\mapsto \diag{a,a^{-1}}:=\tilde{a}$ induces a homomorphism, $\mu$:
\begin{eqnarray*}
\ext{2}{F^\times}\cong \hoz{2}{F^\times}&\to&\hoz{2}{\specl{2}{F}},\\
 a\wedge b &\mapsto& 
\left([\tilde{a}|\tilde{b}]-[\tilde{b}|\tilde{a}]\right)\otimes 1\in B_2(\specl{2}{F})\otimes_{\gr{\Z}{\specl{2}{F}}}\Z.
\end{eqnarray*}

Mazzoleni's calculations (see \cite{mazz:sus}, Lemma 6)
 show that $\mu(\ext{2}{F^\times})=\ker{\omega}$ and that 
there is an isomorphism $\mu(\ext{2}{F^\times})\cong 2\cdot\milk{2}{F}$ given by
$\mu(a\wedge b)\mapsto 2\{ a,b\}$.

On the other hand, a straightforward calculation shows that 
\[
\ee{2}\left(\mu(a\wedge b)\right) = 
\fgen{a}\ssb{b,\frac{1}{ab}}-\ssb{b,\frac{1}{b}}-\fgen{a}\ssb{1,\frac{1}{a}}+
\fgen{b}\ssb{1,\frac{1}{b}}+\ssb{a,\frac{1}{a}}-\fgen{b}\ssb{a,\frac{1}{ab}}:=C_{a,b}
\]

Now by the diagram above, 
\[
\T{2}(C_{a,b})=\T{2}(\ee{2}\left(\mu(a\wedge b)\right) )\in
\ker{\mwk{2}{F}\to I^2(F)} \cong 2\milk{2}{F}.
\]
  Recall that the natural embedding 
$2\milk{2}{F}\to \mwk{2}{F}$ is given by $2\{ a,b\}\mapsto [a^2][b]=[a][b]-[b][a]$ and the composite
\[
\xymatrix{2\milk{2}{F}\ar[r]&\mwk{2}{F}\ar[r]^-{\kk{2}}&\milk{2}{F}}
\]
is the natural inclusion map.   Since 
\begin{eqnarray*}
\kk{2}\left(\T{2}(C_{a,b})\right) &=&
\left\{ b,\frac{1}{ab}\right\}-\left\{ b,\frac{1}{b}\right\}-\left\{ 1,\frac{1}{a}\right\}+
\left\{ 1,\frac{1}{b}\right\}+\left\{ a,\frac{1}{a}\right\}-\left\{ a,\frac{1}{ab}\right\}\\
&=& \{ a,b\}-\{ b,a\}=2\{ a,b\},
\end{eqnarray*}
it follows that we have a commutative diagram  with exact rows
\[
\xymatrix{
0\ar[r]
&
\mu(\ext{2}{F^\times})\ar[r]\ar[d]^-{\cong}
&
\hoz{2}{\specl{2}{F}}\ar[r]^-{\omega}\ar[d]^-{\T{2}\circ\ee{2}}
&
I^2(F)\ar[r]\ar[d]^-{=}
&
0\\
0\ar[r]
&
2\milk{2}{F}\ar[r]
&
\mwk{2}{F}\ar[r]
&
I^2(F)\ar[r]
&
0
}
\]
proving the theorem. 
\end{proof}

\section{$\am$-modules}\label{sec:am}

From the results of the last section, it follows that there is a $\gr{\Z}{F^\times}$-decomposition 
\[
\SSn{F^2}\cong \mwk{2}{F}\oplus \gr{\Z}{F^\times}\oplus ?
\]

It is not difficult to determine that the missing factor is isomorphic to the $1$-dimensional vector space $F$  
(with the tautological $F^\times$-action). 
However, as we will see, this extra term will not play any role in the calculations of  $\hoz{n}{\specl{k}{F}}$. 

As  $\grf{F}$-modules, our main objects of interest (Milnor-Witt $K$-theory, the homology of the special linear group, 
the powers of the fundamental ideal in the Grothendieck-Witt ring) are what we call below `\mult'; there exists $m\geq 1$ 
such that, for all $a\in F^\times$, $\fgen{a^m}$ acts trivially. This is certainly not true of the vector space $F$ above.  
In this section we formalise this difference, and use this formalism to prove an analogue of Suslin's Theorem 1.8 (\cite{sus:homgln}) 
(see 
Theorem \ref{thm:speclam} below).

%\emph{Throughout the remainder of this paper,
% $F$ is a field of characteristic $0$. The prime subfield will be denoted $k$; i.e. $k=\Q$.}

%\subsection{Definitions and examples}
\emph{Throughout the remainder of this article, $F$ will denote a field of characteristic $0$.}

Let $\mset{F}\subset\grf{F}$ denote the multiplicative set generated by the elements\\
 $\{ \ffist{a}=\fgen{a}-1\ |\ a\in F^\times\setminus
\{ 1\}\}$. Note that $0\not\in\mset{F}$, since the elements of $\mset{F}$ map to units under the natural 
ring homomorphism $\grf{F}\to F$. We will also let $\msetp{\Q}\subset\grf{\Q}$ denote the multiplicative set generated by 
 $\{ \ffist{a}=\fgen{a}-1\ |\ a\in \Q^\times\setminus
\{ \pm 1\}\}$.

\begin{defi} A $\grf{F}$-module $M$ is said to be \emph{\mult} if there exists 
$s\in\msetp{\Q}$ with $sM=0$.
\end{defi}
\begin{defi}
We will say that a $\grf{F}$-module is \emph{\add} if every $s\in \msetp{\Q}$ acts as an automorphism on $M$.  
\end{defi}
\begin{exa}
Any trivial $\grf{F}$-module $M$ is \mult, since $\ffist{a}$ annihilates $M$ for all $a\not=1$.
\end{exa}
\begin{exa}
$\gw{F}$, and more generally $I^n(F)$, is \mult since $\ffist{a^2}$ annihilates these modules for all $a\in F^\times$.
\end{exa}

\begin{exa} 
Similarly, the groups $\hoz{n}{\specl{n}{F}}$ are \mult since they are annihilated by the elements $\ffist{a^m}$.
\end{exa}

\begin{exa}
Any vector space over $F$, with the induced action of $\grf{F}$, is \add since all elements of $\mset{F}$ act as automorphisms.
\end{exa}

\begin{exa} \label{exa:addtens}
More generally, if $V$ is a vector space over $F$, then for all $r\geq 1$, 
the $r$th tensor power $\tens{r}{\Z}{V}=\tens{r}{\Q}{V}$ is an \add 
module since, if $a\in \Q\setminus\{\pm 1\}$, $\fgen{a}$ acts as multiplication by $a^r$ and hence $\ffist{a}$ acts as multiplication 
by $a^r-1$. For the same reasons, the $r$th exterior power, $\extr{r}{\Z}{V}$, is an \add module.   
\end{exa}

\begin{rem} Observe that if $\ffist{a^m}$ acts as an automorphism of the $\grf{F}$-module $M$ for some $a\in F^\times$, $m>1$, then 
so does $\ffist{a}$, since $\ffist{a^m}=\ffist{a}(\fgen{a^{m-1}}+\cdots +\fgen{a}+1)=(\fgen{a^{m-1}}+\cdots +\fgen{a}+1)\ffist{a}$ in
$\grf{F}$. 
\end{rem}

%\subsection{Results on $\am$ modules}

\begin{lem}\label{lem:multses}
Let
\[
0\to M_1\to M\to M_2\to 0
\]
be a short exact sequence of $\grf{F}$-modules. 

Then $M$ is \mult if and only if  $M_1$ and $M_2$ are.
\end{lem}
\begin{proof}
Suppose $M$ is \mult. If  $s\in \msetp{\Q}$ satisfies $sM=0$, it follows that 
$sM_1=sM_2=0$.

Conversely, if $M_1$ and $M_2$ are \mult then there exist $s_1,s_2\in\msetp{\Q}$ with 
$s_iM_i=0$ for $i=1,2$. It follows that $sM=0$ for $s=s_1s_2\in \msetp{\Q}$. 
\end{proof}

\begin{lem}\label{lem:addses}
Let
\[
0\to A_1\to A\to A_2\to 0
\] 
be a short exact sequence of $\grf{F}$-modules. If $A_1$ and $A_2$ are \add modules, then so is $A$.
\end{lem}
\begin{proof}
This is immediate from the definition.
\end{proof}
\begin{lem} \label{lem:kercoker}
Let $\phi:M\to N$ be a homomorphism of $\grf{F}$-modules.
\begin{enumerate}
\item If $M$ and $N$ are \mult, then so are $\ker{\phi}$ and $\coker{\phi}$.
\item If $M$ and $N$ are \add, then so are $\ker{\phi}$ and $\coker{\phi}$.
\end{enumerate}
\end{lem}
\begin{proof}
\begin{enumerate}
\item This follows from Lemma \ref{lem:multses} above.
\item  If $s\in \msetp{\Q}$, then $s$ acts as an 
automorphism of $M$ and $N$, and hence of $\coker{\phi}$ and $\ker{\phi}$.
\end{enumerate}
\end{proof}
\begin{cor}\label{cor:homol} 
Let $C=(C_\bullet, d)$ be a complex of $\grf{F}$-modules. If $C_\bullet$ is \add (i.e. if each $C_n$ is an \add module),
 then each $H_n(C)$ is 
an $\add$ module. If each $C_n$ is \mult then each $H_n(C)$ is a \mult module. 
\end{cor}

\begin{lem}\label{lem:am0}
Let $M$ be a \mult $\grf{F}$-module and $A$ an \add $\grf{F}$-module. 
Then $\hom{\grf{F}}{M}{A}=0$ and $\hom{\grf{F}}{A}{M}=0$.
\end{lem}
\begin{proof}
Let $f:M\to A$ be a $\grf{F}$-homomorphism. Every 
$s\in \msetp{\Q}$  acts as an automorphism of $A$. However, there exists $s\in \msetp{\Q}$ 
with $sM=0$. Thus, for $m\in M$, $0=f(sm)=sf(m)\imp f(m)=0$.

Let $g:A\to M$ be a $\grf{F}$-homomorphism. Again, choose $s\in \msetp{\Q}$ acting as an automorphism of $A$ and annihilating 
$M$. If $a\in A$, then there exists $b\in a$ with $a=sb$. Hence $g(a)=sg(b)=0$ in $M$.
\end{proof}
\begin{lem}
If $P$ is a $\grf{F}$-module and if $A$ is an \add submodule and $M$ a \mult submodule, then $A\cap M=0$.
\end{lem}
\begin{proof}
There exists $s\in\grf{\Q}$ which annihilates any submodule of $M$ but is injective on any submodule of $A$.  
\end{proof}

\begin{lem}\label{lem:split}\ 

\begin{enumerate}
\item If 
\[
\xymatrix{
0\ar[r]&M\ar[r]&H\ar[r]^-{\pi}&A\ar[r]&0
}
\]
is an exact sequence of $\grf{F}$-modules with $M$  \mult \ and $A$ \add then the sequence splits (over $\grf{F}$).
\item Similarly, if 
\[
\xymatrix{
0\ar[r]&A\ar[r]&H\ar[r]&M\ar[r]&0
}
\]
is an exact sequence of $\grf{F}$-modules with $M$  \mult \ and $A$ \add then the sequence splits.
\end{enumerate}
\end{lem}
\begin{proof}\ 
As above we can find $s\in \grf{\Q}$ such that $s\cdot M=0$ and $s$ acts as an 
automorphism of $A$.
\begin{enumerate}
\item  
Then $s H$ is a $\grf{F}$-submodule of $H$ and $\pi$ induces an isomorphism 
$s H\cong A$, since $\pi(sH)=s\pi(H)=sA=A$ and if $\pi(sh)=0$ then $s\pi(h)=0$ in $A$, so that $\pi(h)=0$ and $h\in M$.

\item  We have  $s H=A$ and multiplication by $s$ gives an automorphism, $\alpha$, of $A$. Thus the 
$\grf{F}$-homomorphism 
$H\to A, h\mapsto \alpha^{-1}(s\cdot h)$ splits the sequence. 
\end{enumerate}
\end{proof}

\begin{defi} We will say that a $\grf{F}$-module $H$ is \emph{an $\am$ module} 
if there exists a \mult
 $\grf{F}$-module $M$ and an \add $\grf{F}$ module $A$  and 
 an isomorphism of $\grf{F}$-modules $H\cong A\oplus M$. 
 \end{defi}

 \begin{lem}
 Let $H$ be an $\am$ module and let  
 \( \phi:H\to A\oplus M\) be an isomorphism of $\grf{F}$-modules, with $M$ \mult and 
 $A$ \add.
 
 Then 
 \[
 \phi^{-1}(A)=\bigcup_{A'\subset H, A' \mathrm{\add}}A'\qquad \mbox{ and }\qquad
  \phi^{-1}(M)=\bigcup_{M'\subset H, M' \mathrm{\mult}}M'
 \]
 \end{lem} 
 \begin{proof}
 Let $M' \subset H$ be multiplicative. Then the composite
 \[
 \xymatrix{
 M'\ar[r]&H\ar[r]^-{\phi}&A\oplus M\ar[r]&A
 }
 \]
 is zero by Lemma \ref{lem:am0}, and thus $M'\subset \phi^{-1}(M)$.

An analogous argument can be applied to $\phi^{-1}(A)$. 
\end{proof}
 
It follows that the submodules $\phi^{-1}(A)$ and $\phi^{-1}(M)$ are independent of 
the choice of $\phi$, $A$ and $M$. 
We will denote the first as $\ad{H}$ and the second 
 as $\mul{H}$. 
 
 Thus if $H$ is an $\am$ module then there is a canonical decomposition 
 $H=\ad{H}\oplus\mul{H}$, where $\ad{H}$ (resp. $\mul{H}$) is the maximal \add (resp. 
 \mult ) submodule of $H$. We have canonical projections 
 \[
 \ad{\pi}:H\to\ad{H}, \qquad \mul{\pi}:H\to\mul{H}.
 \]
 
 \begin{lem}\label{lem:R}
 Let $H$ be a $\am$ module. Suppose that $H$ is also a module over 
 a ring $R$  and that the action of $R$ commutes with that of $\grf{F}$. Then 
 $\ad{H}$ and $\mul{H}$ are $R$-submodules of $H$. 
 \end{lem}
 \begin{proof}
 Let $r\in R$. Then the composite 
 \[
 \xymatrix{
 \ad{H}\ar[r]^-{r\cdot}&H\ar[r]^-{\mul{\pi}}&\mul{H}
 }
 \]
 is a $\grf{F}$-homomorphism and thus is $0$ by Lemma \ref{lem:am0}. It follows that 
 $r\cdot \ad{H}\subset \ker{\mul{\pi}}=\ad{H}$.
 \end{proof}
 
\begin{lem}\label{lem:oplus}
 Let $f:H\to H'$ be a $\grf{F}$-homomorphism of $\am$ modules. 

Then there exist $\grf{F}$-homomorphisms $\ad{f}:\ad{H}\to\ad{H'}$ and 
$\mul{f}:\mul{H}\to\mul{H'}$ such that $f=\ad{f}\oplus\mul{f}$.

Suppose that  $H$ and $H'$ are modules over a ring $R$ and that the $R$-action commutes 
with the $\grf{F}$-action in each case. If $f$ is an $R$-homomorphism, then so are 
$\ad{f}$ and $\mul{f}$.
\end{lem} 
\begin{proof}
This is immediate from Lemmas \ref{lem:am0} and \ref{lem:R}.
\end{proof}

\begin{lem}\label{lem:am}
If
\[
\xymatrix{
0\ar[r]&L\ar[r]^-{j}&H\ar[r]^-{\pi}&K\ar[r]&0
}
\]
is a short exact sequence of $\grf{F}$-modules and if $L$ and $K$ are $\am$ modules, then so is $H$.
\end{lem}
\begin{proof}
Let $\tilde{H}=\pi^{-1}(\mul{K})$.
Then the exact sequence
\[
0\to L\to \tilde{H}\to \mul{K}\to 0
\] 
gives the exact sequence
\[
0\to \frac{L}{\mul{L}}\to \frac{\tilde{H}}{j(\mul{L})}\to \mul{K}\to 0.
\]
Since $L/\mul{L}\cong\ad{L}$ is \add, this latter sequence is split, by Lemma \ref{lem:split} (2).

So $\tilde{H}/j(\mul{L})$ is a $\am$ module, and there is a $\grf{k}$-isomorphism
\[
\xymatrix{
\tilde{H}/j(\mul{L})\ar[r]^-{\phi}_-{\cong}&\ad{L}\oplus\mul{K}.
}
\] 
Let $\bar{\phi}$ be the composite
\[
\xymatrix{
\tilde{H}\ar[r]&\tilde{H}/j(\mul{L})\ar[r]^-{\phi}&\ad{L}\oplus\mul{K}.
}
\]

Let $H_m=\bar{\phi}^{-1}(\mul{K}) \subset \tilde{H}\subset H$. Then, we have an exact sequence
\[
0\to\mul{L}\to H_m\to\mul{K}\to 0
\]
so that $H_m$ is \mult. 

On the other hand, since $\tilde{H}/H_m\cong\ad{L}$ and $H/\tilde{H}\cong \ad{K}$, 
we have a short exact sequence
\[
0\to \ad{L}\to \frac{H}{H_m}\to \ad{K}\to 0. 
\]
This implies that $H/H_m$ is \add, and thus $H$ is $\am$ by Lemma \ref{lem:split} (1). 
\end{proof}

\begin{lem} \label{lem:homol} Let $(C_\bullet,d)$ be a complex of $\grf{k}$-modules. If each $C_n$ is $\am$, then $\homol{\bullet}{C}$ is 
$\am$, and furthermore
\begin{eqnarray*}
\homol{\bullet}{\ad{C}}=\ad{\homol{\bullet}{C}}\\
\homol{\bullet}{\mul{C}}=\mul{\homol{\bullet}{C}}\\
\end{eqnarray*}
\end{lem}
\begin{proof}
The differentials $d$ decompose as $d=\ad{d}\oplus\mul{d}$ by Lemma \ref{lem:oplus}.
\end{proof}

\begin{thm}\label{thm:eram}
Let $(E^r,d^r)$ be a first quadrant spectral sequence of $\grf{k}$-modules converging to the $\grf{k}$-module 
$H_\bullet=\{ H_n\}_{n\geq 0}$.

If for some $r_0\geq 1$ all of the modules $E^{r_0}_{p,q}$ are $\am$, then the same holds for all the modules $E^r_{p,q}$ for 
all $r\geq r_0$ and hence for the modules $E^\infty_{p,q}$.

Furthermore, $H_\bullet$ is $\am$ and the spectral sequence decomposes as a direct sum $E^r=\ad{E^r}\oplus\mul{E^r}$ ($r\geq r_0$)
with $\ad{E^r}$ converging to $\ad{H_\bullet}$ and $\mul{E^r}$ converging to $\mul{H_\bullet}$.
\end{thm}

\begin{proof}
Since $E^{r+1}=\homol{}{E^r,d^r}$ for all $r$, the first statement follows from Lemma \ref{lem:homol}. 

Since $E^r$ is a first quadrant 
spectral sequence (and, in particular, is bounded), it follows that for any fixed $(p,q)$, $E^\infty_{p,q}=E^r_{p,q}$ 
for all sufficiently large $r$. Thus $E^\infty$ is also $\am$. 

Now $H_n$ admits a filtration $0=\filt{0}{H_n}\subset\cdots\subset\filt{n}{H_n}=H_n$ with corresponding 
quotients $\grd{p}{H_n}\cong E^\infty_{p,n-p}$. 

Since all the quotients are $\am$, it follows by Lemma \ref{lem:am}, together with an induction on the filtration length, that $H_n$ 
is $\am$.

The final two statements follow again from Lemma \ref{lem:homol}. 
\end{proof}

If $G$ is a subgroup of $\gnl{V}$, we let $SG$ denote $G\cap\spcl{V}$.

\begin{thm}\label{thm:speclam}
Let $V$, $W$ be finite-dimensional vector spaces over $F$ and let 
$G_1\subset \gnl{W}$, $G_2\subset\gnl{V}$ be subgroups and suppose that $G_2$ contains the group $F^\times$ of 
scalar matrices.

Let $M$ be a subspace of $\hom{F}{V}{W}$ for which $G_1M=M=MG_2$.  

Let 
\[
G=\begin{pmatrix}
G_1&M\\
0&G_2
\end{pmatrix}
\subset \gnl{W\oplus V}.
\]

Then, for $i\geq 1$,  the groups $\hoz{i}{SG}$ are $\am$ and the natural embedding $j:S(G_1\times G_2)\to SG$ induces an isomorphism
\[
\hoz{i}{S(G_1\times G_2)}\cong \mul{\hoz{i}{SG}}.
\]
\end{thm}

\begin{proof}
We begin by noting that the groups $\hoz{i}{SG}$ are $\grf{F}$-modules: 
The action of $F^\times$ is derived from the short exact sequence
\[\xymatrix{
1\ar[r]& SG\ar[r]& G\ar[r]^-{\det}&F^\times\ar[r]&1
}
\]
We have a split extension of groups (split by the map $j$) which is $F^\times$-stable:
\[
\xymatrix{
0\ar[r]&M\ar[r]&SG\ar[r]^-{\pi}&S(G_1\times G_2)\ar[r]&1.
}
\]
The resulting Hochschild-Serre spectral sequence has the form 
\[
E^2_{p,q}=\ho{p}{S(G_1\times G_2)}{\hoz{q}{M}}\Longrightarrow \hoz{p+q}{SG}.
\]  
This spectral sequence exists in the category of $\grf{F}$-modules and  all differentials and edge homomorphisms 
are $\grf{F}$-maps.

Since the map $\pi$ is split by $j$ it induces a split surjection on integral homology groups. Thus 
\[
\hoz{n}{S(G_1\times G_2)}=E^2_{n,0}=E^\infty_{n,0} \quad \mbox{ for all } n\geq 0.
\]

Observe furthermore that the $\grf{F}$-module $\hoz{n}{S(G_1\times G_2)}$ is \mult: Given $a\in F^\times$, the element 
\[
\rho_a:=
\begin{pmatrix}
\id{W}&0\\
0&a\cdot\id{V}
\end{pmatrix}
\in G
\]
has determinant $a^m$ ( $m=\dim{F}{V}$) 
and centralizes $S(G_1\times G_2)$. It follows that $\fgen{a^m}$ acts trivially on $\hoz{n}{S(G_1\times G_2)}$ 
for all $n$; i.e. $\ffist{a^m}$ annihilates $\hoz{n}{S(G_1\times G_2)}$.

%Now $M$ is a finite-dimensional vector space over $F$. Thus,the groups $\hoz{q}{M}$, with $F^\times$ acting by scalar multiplication, 
%are \add\  by Theorem \ref{thm:homv}.

Recall (Example \ref{exa:addtens} above) that for $q\geq 1$, the modules $\hoz{q}{M}=\extr{q}{\Z}{M}$, with the $\grf{F}$-action derived 
from the action of $F$ by scalars on $M$, are \add modules. 

Now if $a\in F^\times$, then conjugation by $\rho_a$ is trivial on $S(G_1\times G_2)$ but acts on $M$ as scalar multiplication by 
$a$. It follows that for $q>0$, $\ffist{a^m}$ acts as an automorphism on $\ho{p}{S(G_1\times G_2)}{\hoz{q}{M}}$ for all 
$a\in \Q\setminus \{\pm 1\}$. Thus, for $q>0$,
the groups $\ho{p}{S(G_1\times G_2)}{\hoz{q}{M}}$  are \add $\grf{F}$-modules; i.e., 
all $E^2_{p,q}$ are \add\  for $q>0$. It follows at once that the groups $E^\infty_{p,q}$ are \add\ 
for all $q>0$. Thus, from the convergence of the spectral sequence,  we have a short exact sequence
\[
0\to H\to \hoz{n}{SG}\to E^\infty_{n,0}=j\left(\hoz{n}{S(G_1\times G_2)}\right)\to 0
\]   
and $H$ has a filtration whose graded quotients are all \add. 

So $\hoz{n}{SG}$ is $\am$ as claimed, and $\mul{\hoz{n}{SG}}\cong\hoz{n}{S(G_1\times G_2)}$.
  
\end{proof}

\begin{cor}
Suppose that $W'\subset W$. Then there is a corresponding inclusion $\saff{W'}{V} \to \saff{W}{V}$. This inclusion induces an 
isomorphism
\[
\xymatrix{
\mul{\hoz{n}{\saff{W'}{V}}} \ar[r]_-{\cong}&\mul{\hoz{n}{\saff{W}{V}}}\cong\hoz{n}{\spcl{V}}\\
}
\]  
for all $n\geq 1$. 
\end{cor}

\section{The spectral sequences}\label{sec:ss}

\textit{Recall that $F$ is a field of characteristic $0$ throughout this section.}

In this section we use the complexes $\cgen{\bullet}{W,V}$ to construct spectral sequences converging to $0$ in dimensions less than 
$n=\dim{F}{V}$, and to $\SSn{W,V}$ in dimension $n$. By projecting onto the \mult part, we obtain spectral sequences with good 
properties: the terms in the $E^1$-page are just the kernels and cokernels of the stabilization maps 
$f_{t,n}:\hoz{n}{\specl{t}{F}}\to\hoz{n}{\specl{t+1}{F}}$. We then prove that the higher differentials are all zero. Since 
the spectral sequences converge to $0$ in low degrees, this already implies the main stability result (Corollary \ref{cor:stab}); the 
maps $f_{t,n}$ are isomorphisms for $t\geq n+1$ and are surjective for $t=n$. The remainder of the paper is devoted to an analysis 
of the case $t=n-1$, which requires some more delicate calculations.

Let $\compt{\bullet}{W,V}$ denote the truncated complex.
\[
\compt{p}{W,V}=
\left\{
\begin{array}{cc}
\cgen{p}{W,V},& p\leq \dim{F}{V}\\
0,& p>\dim{F}{V}
\end{array}
\right.
\]  

Thus 
\[
\homol{p}{\compt{\bullet}{W,V}}=
\left\{
\begin{array}{ll}
0,& p\not=n\\
H(W,V),& p=n
\end{array}
\right.
\]
where $n=\dim{F}{V}$.

Thus the  natural action of $\saff{W}{V}$ on $\compt{\bullet}{W,V}$ gives rise to a spectral sequence $\EE{W}{V}$
 which has the form
\[
E^1_{p,q}=\ho{p}{\saff{W}{V}}{\compt{q}{W,V}}\Longrightarrow \ho{p+q-n}{\saff{W}{V}}{H(W,V)}.
\]

The groups $\compt{q}{W,V}$ are permutation modules for $\saff{W}{V}$ and thus the $E^1$-terms (and the differentials $d^1$) can be 
computed in terms of the homology of stabilizers.

Fix a basis $\{ e_1,\ldots,e_n\}$ of $V$. Let $V_r$ be the span of $\{ e_1,\ldots,e_r\}$ and let $V'_s$ be the span of 
$\{ e_{n-s},\ldots,e_n\}$, so that $V=V_r\oplus V'_{n-r}$ if $0\leq r\leq n$.

For any $0\leq q\leq n-1$, the group $\saff{W}{V}$ acts transitively on the basis of $\compt{q}{W,V}$ and the stabilizer of 
\[
\big((0,e_1),\ldots,(0,e_q)\big)
\] 
is $\saff{W\oplus V_q}{V'_{n-q}}$.

Thus, for $q\leq n-1$,
\[
E^1_{p,q}= \ho{p}{\saff{W}{V}}{\compt{q}{W,V}}\cong \hoz{p}{\saff{W\oplus V_q}{V'_{n-q}}}
\] 
by Shapiro's Lemma.

By the results in section \ref{sec:am} we have:

\begin{lem}\label{lem:eeam}
 The terms $E^1_{p,q}$ in the spectral sequence $\EE{W}{V}$ are $\am$ for $q>0$, and 

\[
\mul{(E^1_{p,q})}=\hoz{p}{\spcl{V'_{n-q}}}\cong\hoz{p}{\specl{n-q}{F}}.
\]
\end{lem}

For $q=n$, the orbits of $\saff{W}{V}$ on the basis of $\compt{n}{W,V}$ are in bijective correspondence with $F^\times$ via 

\[
\big((w_1,v_1),\ldots,(w_n,v_n)\big)\mapsto \det\left([v_1|\cdots|v_n]_{\mathcal{E}}
\right).
\]

The stabilizer of any basis element of $\compt{n}{W,V}$ is trivial. Thus
\[
E^1_{p,n}=
\left\{
\begin{array}{ll}
\grf{F},& p=0\\
0,& p>0
\end{array}
\right.
\]

Of course, $E^1_{p,q}=0$ for $q>n$.

The first column of the $E^1$-page of the spectral sequence $\EE{W}{V}$ has the form 
\[
E^1_{0,q}=\left\{
\begin{array}{ll}
\Z,& q<n\\
\grf{F},&q=n\\
0,& q>n
\end{array}
\right.
\]
and the differentials are easily computed:
For $q<n$
\[
d^1_{0,q}:E^1_{0,q}\to E^1_{0,q} = 
\left\{
\begin{array}{ll}
\id{\Z},& q \mbox{ is odd }\\
0,& q \mbox{ is even }
\end{array}
\right.
\]
and
\[
d^1_{0,n}:\grf{F}\to \Z=
\left\{
\begin{array}{ll}
\mbox{augmentation },& n \mbox{ odd }\\
0,& n \mbox{ even }
\end{array}
\right.
\]

It follows that $E^2_{0,q}=0$ for $q\not= n$ and 
\[
E^2_{0,n}=
\left\{
\begin{array}{ll}
\aug{F^\times},& n \mbox{ odd }\\
\grf{F},& n \mbox{ even }
\end{array}
\right. 
\]

Note that the composite
\[
\xymatrix{
\SSS(W,V)\ar[r]^-{\mbox{\tiny edge}}&E^\infty_{0,n}\subset E^2_{0,n}=\Aa{n}
}
\]
is just the map $\Dd{W,V}$ of section \ref{sec:prelim} above.

\begin{lem}
The map $\Dd{W,V}$ is a split surjective homomorphism of $\grf{F}$-modules.
\end{lem}

\begin{proof}
If $W=0$, this is Lemma \ref{lem:dn} (1) and (3), since $V\cong F^n$. 

In general the natural map of complexes $\compt{\bullet}{V}\to\compt{\bullet}{W,V}$ gives rise to a commutative diagram of 
$\grf{F}$-modules
\[
\xymatrix{
\SSS(V)\ar[rr]\ar[dr]_-{\Dd{V}}&&\SSS(W,V)\ar[dl]^-{\Dd{W,V}}\\
&\Aa{n}&
}
\]
\end{proof}

We let $\SSp{W,V}:=\ker{\Dd{W,V}:\SSS(W,V)\to \Aa{n}}$, so that $\SSS(W,V)\cong \SSp{W,V}\oplus \Aa{n}$ for all $W, V$.

\begin{cor} In the spectral sequence $\EE{W}{V}$, we have $E^2_{0,q}=E^\infty_{0,q}$ for all $q\geq 0$.

All higher differentials $d^r_{0,q}:E^r_{0,q}\to E^r_{r-1,q+r}$ are zero.
\end{cor}

It follows that the spectral sequences $\EE{W}{V}$ decompose as a direct sum of two spectral sequences 
\[
\EE{W}{V}=\EEz{W}{V}\oplus \EEp{W}{V}
\]
where $\EEz{W}{V}$ is the first column of $\EE{W}{V}$ and $\EEp{W}{V}$ involves only the terms $E^r_{p,q}$ with $q>0$.

The spectral sequence $\EEz{W}{V}$ converges in degree $d$ to 
\[
\left\{
\begin{array}{ll}
0,& d\not= n\\
\Aa{n},& d=n
\end{array}
\right.
\] 

The spectral sequence $\EEp{W}{V}$ converges in degree $d$ to 
\[
\left\{
\begin{array}{ll}
0,& d<n\\
\SSp{W,V},& d=n\\
\ho{d-n}{\saff{W}{V}}{H(W,V)},&d>n
\end{array}
\right.
\] 

By Lemma \ref{lem:eeam} above, all the terms of the spectral sequence $\EEp{W}{V}$ are $\am$. We thus have

\begin{cor}\label{cor:sfam}\ 

\begin{enumerate}
\item The $\grf{F}$-modules $\SSp{W,V}$ are $\am$.
\item The graded submodule $\ad{\SSp{F^\bullet}}\subset \SSS(F^\bullet)$ is an ideal.
\end{enumerate}
\end{cor}

\begin{proof}\ 

\begin{enumerate}
\item This follows from Theorem \ref{thm:eram}.
\item This follows from Lemma \ref{lem:R}, since $\SSp{F^\bullet}$ is an ideal in $\SSS(F^\bullet)$ by Lemma \ref{lem:dn} (2). 
\end{enumerate}
\end{proof}

\begin{cor}\label{cor:ssp}
The natural embedding $H(V)\to H(W,V)$ induces an isomorphism
\[
\xymatrix{
\mul{\SSp{V}}\ar[r]^-{\cong}&\mul{\SSp{W,V}}.
}
\]
\end{cor}

\begin{proof}
The map of complexes of $\spcl{V}$-modules  $\compt{\bullet}{V}\to \compt{\bullet}{W,V}$ gives rise to a map of spectral 
sequences $\EEEp{V}\to\EEp{W}{V}$ and hence a map $\mul{\EEEp{V}}\to\mul{\EEp{W}{V}}$. The induced map on the $E^1$-terms is
\[
\xymatrix{
\hoz{p}{\specl{n-q}{F}}\ar[r]^-{\id{}}\ar[d]^-{\cong}&\hoz{p}{\specl{n-q}{F}}\ar[d]^-{\cong}\\
\mul{\ho{p}{\spcl{V}}{\compt{q}{V}}}\ar[r]&\mul{\ho{p}{\saff{W}{V}}{\compt{q}{W,V}}}
}
\]
and thus is an isomorphism. 

It follows that there is an induced isomorphism of abutments
\[
\mul{\SSp{V}}\cong \mul{\SSp{W,V}}
\]
and
\[
\mul{\ho{k}{\spcl{V}}{H(V)}}\cong \mul{\ho{k}{\saff{W}{V}}{H(W,V)}}.
\]
\end{proof}

For convenience, we now \emph{define} 
\[
\mul{\SSS(W,V)}:=\frac{\SSS(W,V)}{\ad{\SSp{W,V}}}
\]
(even though $\SSS(W,V)$ is not an $\am$ module).

This gives:

\begin{cor}\label{cor:isom}
\begin{eqnarray*}%\label{eqn:isom}
\mul{\SSS(W,V)}\cong \mul{\SSp{W,V}}\oplus \Aa{n}\cong \mul{\SSp{V}}\oplus \Aa{n} \cong \mul{\SSS(V)}
\end{eqnarray*}
as $\grf{F}$-modules, and $\mul{\SSS(F^\bullet)}$ is a graded $\grf{F}$-algebra. 
\end{cor}

\begin{lem}\label{lem:corinv}
For any $k\geq 1$, the corestriction map 
\[
\cores{}{}:\hoz{i}{\specl{k}{F}}\to\hoz{i}{\specl{k+1}{F}}
\]
is $F^\times$-invariant;i.e. if $a\in F^\times$ and $z\in \hoz{i}{\specl{k}{F}}$, then 
\[
\cores{}{}(\fgen{a}z)=\fgen{a}\cores{}{}(z)=\cores{}{}(z).
\]
\end{lem}

\begin{proof}
Of course, $\cores{}{}$ is a homomorphism of $\grf{F}$-modules. However, for $a\in F^\times$, $\fgen{a^k}$ acts trivially on 
$\hoz{i}{\specl{k}{F}}$ while $\fgen{a^{k+1}}$ acts trivially on $\hoz{i}{\specl{k+1}{F}}$ so that 
\[
\cores{}{}(\fgen{a}z)=\cores{}{}(\fgen{a^{k+1}}z)=\fgen{a^{k+1}}\cores{}{}(z)=\cores{}{}(z).
\]
\end{proof}

\begin{lem}\label{lem:d1}
For $0\leq q < n$, the differentials of the spectral sequence $\mul{\EEp{W}{V}}$
\[
d^1_{p,q}:\mul{(E^1_{p,q})}\cong\hoz{p}{\specl{n-q}{F}}
\to \mul{(E^1_{p,q-1})}\cong
\hoz{p}{\specl{n-q+1}{F}}
\]
are zero when $q$ is even and are equal to the corestriction map when $q$ is odd.
\end{lem}

\begin{proof}
$d^1$ is derived from the map $d_q:\compt{q}{W,V}\to\compt{q-1}{W,V}$ of permutation modules. Here
\begin{eqnarray*}
d_q\big((0,e_1),\ldots,(0,e_q)\big)
&=& 
\sum_{i=1}^q (-1)^{i+1}
\big((0,e_1),\ldots,\widehat{(0,e_i)},\ldots,(0,e_q)\big)\\
&=&
\sum_{i=1}^q(-1)^{i+1}\phi_i\big((0,e_1),\ldots,(0,e_{q-1})\big)
\end{eqnarray*}
where $\phi_i\in \saff{W}{V}$ can be chosen to be of the form 
\[
\phi_i=
\begin{pmatrix}
\id{W}&0\\
0&\psi_i
\end{pmatrix},\quad
\psi_i=
\begin{pmatrix}
\sigma_i&0\\
0&\tau_i
\end{pmatrix}
\in \gnl{V}
\]
with $\sigma_i\in \gnl{V_q}$ a permutation matrix of determinant $\epsilon_i$ and $\tau_i\in\gnl{V'_{n-q}}$ also of determinant 
$\epsilon_i$. 

$\phi_i$ normalises $\saff{W\oplus V_q,V'_{n-q}}$ and $\spcl{V'_{n-q}}$. Thus for $z\in \hoz{p}{\spcl{V'_{n-q}}}$, 
\begin{eqnarray*}
d^1(z)&=&\sum_{i=1}^q(-1)^{i+1}\cores{}{}(\tau_i z)\\
&=& \sum_{i=1}^q(-1)^{i+1}\cores{}{}(\fgen{\epsilon_i}z)\\
&=& \sum_{i=1}^q(-1)^{i+1}\cores{}{}(z)
=\left\{
\begin{array}{ll}
\cores{}{}(z),& q\mbox{ odd}\\
0,& q\mbox{ even}\\
\end{array}
\right.
\end{eqnarray*}
\end{proof}

Let $E:=\ssb{-1,1}\in \mul{\SSS(F^2)}$. $E$ is represented by the element 
\[
\tilde{E}:=d_3(e_1,e_2,e_2-e_1)=(e_2,e_2-e_1)-(e_1,e_2-e_1)+(e_1,e_2)\in H(F^2)\subset \compt{2}{F^2}.
\]

Multiplication by $\tilde{E}$ induces a map of complexes of $\genl{n-2}{F}$-modules
\[
\shift{2}{\compt{\bullet}{F^{n-2}}}\to \compt{\bullet}{F^n}
\]

There is an induced map of spectral sequences $\shift{2}{\EEE{F^{n-2}}}\to\EEE{F^n}$, which in turn induces a map 
$\shift{2}{\EEEp{F^{n-2}}}\to \EEEp{F^n}$, and hence a map $\shift{2}{\mul{\EEEp{F^{n-2}}}}\to \mul{\EEEp{F^n}}$.

By the work above, the $E^1$-page of $\mul{\EEEp{F^n}}$ has the form 
\[
E^1_{p,q}=\hoz{p}{\specl{n-q}{F}} \quad (p>0)
\]
while the $E^1$-page of $\shift{2}{\mul{\EEEp{F^{n-2}}}}$ has the form 
\[
{E'}^1_{p,q}=
\left\{
\begin{array}{ll}
\hoz{p}{\specl{(n-2)-(q-2)}{F}}=\hoz{p}{\specl{n-q}{F}},& q\geq 2, p>0\\
0,& q\leq 1 \mbox{ or } p=0
\end{array}
\right.
\]

\begin{lem}\label{lem:period}
For $q\geq 2$ (and $p>0$), the map 
\[
{E'}^1_{p,q}\cong \hoz{p}{\specl{n-q}{F}}\to E^1_{p,q}=\hoz{p}{\specl{n-q}{F}}
\]
induced by $\tilde{E}\ssprod - $ is the identity map.
\end{lem}

\begin{proof}
There is a commutative diagram
\[
\xymatrix{
{E'}^1_{p,q}=\hoz{p}{\specl{n-q}{F}}\ar[r]\ar[d]^-{\mul{(\tilde{E}\ssprod -)}}
&\hoz{p}{\saff{F^{q-2}}{F^{n-q}}}\ar[r]^-{\cong}\ar[d]^-{\tilde{E}\ssprod -}
&\ho{p}{\specl{n-2}{F}}{\compt{q-2}{F^{n-2}}}\ar[d]^-{\tilde{E}\ssprod -}\\
E^1_{p,q}=\hoz{p}{\specl{n-q}{F}}\ar[r]
&\hoz{p}{\saff{F^q}{F^{n-q}}}\ar[r]^-{\cong}
&\ho{p}{\specl{n}{F}}{\compt{q}{F^n}}
}
\]

We number the standard basis of  $F^{n-2}$ $e_3,\ldots,e_n$ so that the inclusion $\specl{n-2}{F}\to\specl{n}{F}$ has the form
\[
A\mapsto 
\begin{pmatrix}
I_2&0\\
0&A
\end{pmatrix}.
\] 

So we have a commutative diagram of  inclusions of groups 
\[
\xymatrix{
\specl{n-q}{F}\ar[r]\ar[d]^-{=}
&\saff{F^{q-2}}{F^{n-q}}\ar[r]\ar[d]
&\specl{n-2}{F}\ar[d]\\
\specl{n-q}{F}\ar[r]
&\saff{F^q}{F^{n-q}}\ar[r]
&\specl{n}{F}.
}
\]

Let $B_\bullet = B_\bullet(\specl{n}{F})$ be the right bar resolution of $\specl{n}{F}$. We can use it to compute the homology of 
any of the groups occurring in this diagram. 

Suppose now that $q\geq 2$ and we have a class, $w$, in ${E'}^1_{p,q}=\hoz{p}{\specl{n-q}{F}}$ represented by a cycle 
\[
z\otimes 1 \in B_p\otimes_{\gr{\Z}{\specl{n-q}{F}}}\Z.
\]  

Its image in $\ho{p}{\specl{n-2}{F}}{\compt{q-2}{F^{n-2}}}$ is represented by $z\otimes(e_3,\ldots,e_q)$. The image of this in 
$\ho{p}{\specl{n}{F}}{\compt{q}{F^n}} $ is 
\begin{eqnarray*}
z\otimes \left[\tilde{E}\ssprod (e_3,\ldots,e_q)\right]&=& z\otimes\left[(e_2,e_2-e_1,e_3,\ldots)-(e_1,e_2-e_1,e_3,\ldots)+
(e_1,e_2,e_3,\ldots)\right]\\
&=& z\otimes \left[(g_1-g_2+1)(e_1,e_2,e_3,\ldots)\right]\in B_p\otimes_{\gr{\Z}{\specl{n}{F}}}\compt{q}{F^n}
\end{eqnarray*}
where 
\[
g_1=
\begin{pmatrix}
0&-1&0&\hdots&0\\
1&1&0&\hdots&0\\
%\hdotsfor{5}\\
0&0&1&0&\vdots\\
\vdots&\vdots&0&\ddots&0\\
0&0&0&\hdots&1\\
\end{pmatrix},\quad
g_2=
\begin{pmatrix}
1&-1&0&\hdots&0\\
0&1&0&\hdots&0\\
%\hdotsfor{5}\\
0&0&1&0&\vdots\\
\vdots&\vdots&0&\ddots&0\\
0&0&0&\hdots&1\\
\end{pmatrix}
\in\specl{n}{F}.
\]

This corresponds to the element in $\hoz{p}{\specl{n-q}{F}}$ represented by 
\[
z(g_1-g_2+1)\otimes 1 \in B_p\otimes_{\gr{\Z}{\specl{n-q}{F}}}\Z
\]

Since the elements $g_i$ centralize $\specl{n-q}{F}$ it follows that this is $(g_1-g_2+1)\cdot w = w$.
\end{proof}

Recall that the spectral sequence $\mul{\EEEp{F^n}}$ converges in degree $n$ to $\mul{\SSp{F^n}}$. Thus there is a filtration
\[
0=\fil{-1}{n}\subset\fil{0}{n}\subset\fil{1}{n}\subset \cdots \fil{n}{n}=\mul{\SSp{F^n}}
\]
with 
\[
\frac{\fil{i}{n}}{\fil{i-1}{n}}\cong E^\infty_{n-i,i}.
\]

The $E^1$-page of $\mul{\EEEp{F^n}}$ has the form
\begin{eqnarray*}
\xymatrix{0&0&0&\hdots&0\\
0&\hoz{1}{\specl{2}{F}}\ar[d]&\hoz{2}{\specl{2}{F}}\ar[d]&\hdots&\hoz{n}{\specl{2}{F}}\ar[d]\\
\vdots&\vdots\ar[d]^-{\cores{}{}}&\vdots\ar[d]^-{\cores{}{}}&\hdots&\vdots\ar[d]^-{\cores{}{}}\\
0&\hoz{1}{\specl{n-2}{F}}\ar[d]^-{0}&\hoz{2}{\specl{n-2}{F}}\ar[d]^-{0}&\hdots&\hoz{n}{\specl{n-2}{F}}\ar[d]^-{0}\\
0&\hoz{1}{\specl{n-1}{F}}\ar[d]^-{\cores{}{}}&\hoz{2}{\specl{n-1}{F}}\ar[d]^-{\cores{}{}}&\hdots&\hoz{n}{\specl{n-1}{F}}
\ar[d]^-{\cores{}{}}\\
0&\hoz{1}{\specl{n}{F}}&\hoz{2}{\specl{n}{F}}&\hdots&\hoz{n}{\specl{n}{F}}\\
}
\end{eqnarray*}

\begin{thm}\label{thm:main}\ 

\begin{enumerate}
\item The higher differentials $d^2,d^3,\ldots, $ in the spectral sequence $\mul{\EEEp{F^n}}$ are all $0$.
\item $\mul{\SSn{F^{n-2}}}\cong E\ssprod\mul{\SSn{F^{n-2}}}$ and this latter is a direct summand of $\mul{\SSn{F^n}}$.
\end{enumerate}
\end{thm}

\begin{proof}\ 

\begin{enumerate}
\item We will use induction on $n$. For $n\leq 2$ the statement is true for trivial reasons.

On the other hand, if $n>2$, 
by Lemma \ref{lem:period}, the map 
\[
\tilde{E}\ssprod -:\shift{2}{\mul{\EEEp{F^{n-2}}}}\to \mul{\EEEp{F^n}}
\]
 induces an 
isomorphism on $E^1$-terms for $q\geq 2$. By induction (and the fact that ${E'}^1_{p,q}=0$ for $q\leq 1$), the result follows for $n$.  
\item 
The map of spectral sequences $\shift{2}{\mul{\EEEp{F^{n-2}}}}\to \mul{\EEEp{F^n}}$ induces a homomorphism on abutments
\[
\xymatrix{
\mul{\SSp{F^{n-2}}}\ar[r]^-{E\ssprod-}&\mul{\SSp{F^n}}
}
\]
By Lemma \ref{lem:period} again, it follows that the composite
\[
\xymatrix{
\mul{\SSp{F^{n-2}}}\ar[r]^-{E\ssprod-}&\mul{\SSp{F^n}}\ar[r]&\left(\mul{\SSp{F^n}}\right)/{\fil{1}{n}}
}
\]
is an isomorphism.

Thus $\mul{\SSp{F^{n-2}}}\cong E\ssprod\mul{\SSp{F^{n-2}}}$ and
\[
\mul{\SSp{F^n}}\cong \left(E\ssprod\mul{\SSp{F^{n-2}}}\right)\oplus \fil{1}{n}.
\]
\end{enumerate}
\end{proof}

As a corollary we obtain the following general homology stability result for the homology 
of special linear groups:
\begin{cor}\label{cor:stab}
\item The corestriction maps $\hoz{p}{\specl{n-1}{F}}\to\hoz{p}{\specl{n}{F}}$ are isomorphisms for $p< n-1$ and are surjective 
when $p=n-1$.
\end{cor}

\begin{proof}
 Using (1) of  Theorem \ref{thm:main} and Lemma \ref{lem:d1}, we have (for the spectral sequence $\mul{\EEEp{F^n}}$):
\[
E^\infty_{p,q}=E^2_{p,q}=\frac{\ker{d^1}}{\image{d^1}}=
\left\{
\begin{array}{ll}
\ker{\hoz{p}{\specl{n-q}{F}}\to\hoz{p}{\specl{n-q+1}{F}}}& q\mbox{ odd}\\
\coker{\hoz{p}{\specl{n-q-1}{F}}\to\hoz{p}{\specl{n-q}{F}}}&q\mbox{ even}
\end{array}
\right.
\]

But the abutment of the spectral sequence is $0$ in dimensions less than $n$.
It follows that $E^\infty_{p,q}=0$ whenever $p+q\leq n-1$.  
\end{proof}

\begin{rem}
Note that in the spectral sequence $\mul{\EEEp{F^n}}$, 
\[
E^\infty_{n,0}= \coker{\hoz{n}{\specl{n-1}{F}}\to\hoz{n}{\specl{n}{F}}}=\SH{n}{F}.
\]
 Clearly, 
the edge homomorphism $\hoz{n}{\specl{n}{F}}\to E^\infty_{n,0}\to \mul{\SSn{F^n}}$ is just the iterated connecting homomorphism $\ee{n}$ of 
section \ref{sec:prelim} above. Thus we have:
\end{rem}
\begin{cor}
The maps 
\[
\ee{\bullet}:\SH{\bullet}{F}\to\mul{\SSn{F^\bullet}}
\]
define an injective homomorphism of graded $\grf{F}$-algebras.
\end{cor}

\begin{cor}\label{cor:E}
$\mul{\SSn{F^2}}=\fil{1}{2}\oplus \grf{F}E$ and for all $n\geq 3$, 
\[
\mul{\SSn{F^n}}=(E\ssprod\mul{\SSn{F^{n-2}}})\oplus \fil{1}{n}\cong \mul{\SSn{F^{n-2}}}\oplus \fil{1}{n}.
\]
\end{cor}

\begin{proof} Clearly $\mul{\SSp{F^2}}=\fil{2}{1}$, while for $n\geq 3$ we have 
\[
\mul{\SSn{F^n}}=
\left\{
\begin{array}{ll}
\mul{\SSp{F^n}}\oplus \grf{F}E^{\ssprod \frac{n}{2}}& n \mbox{ even}\\
\mul{\SSp{F^n}}\oplus \left(\SSn{F}\ssprod E^{\ssprod \frac{n-1}{2}}\right)& n \mbox{ odd}
\end{array}
\right.
\] 
\end{proof}

\begin{cor} For all $n\geq 3$,
\[
\mul{\SSn{F^n}}\cong
\left\{
\begin{array}{ll}
 \fil{1}{n}\oplus \fil{1}{n-2}\oplus \cdots \oplus \fil{1}{2}\oplus \grf{F}& n \mbox{ even}\\
\fil{1}{n}\oplus \fil{1}{n-2}\oplus \cdots \oplus \fil{1}{3}\oplus \aug{F^\times}& n \mbox{ odd}
\end{array}
\right.
\]
as a $\grf{F}$-module.
\end{cor}

Note that $\fil{1}{1}=\SSn{F}=\aug{F^\times}$, and for all $n\geq 2$,
 $\fil{1}{n}$ fits into an exact sequence associated to the spectral sequence $\mul{\EEEp{F^n}}$:
\[
0\to E^\infty_{n,0}=\fil{0}{n}\to \fil{1}{n}\to E^\infty_{n-1,1}\to 0.
\]

\begin{cor}\label{cor:exact}
For all $n\geq 2$ we have an exact sequence
\[
\hoz{n}{\specl{n-1}{F}}\to\hoz{n}{\specl{n}{F}}\to\fil{1}{n}\to \hoz{n-1}{\specl{n-1}{F}}\to\hoz{n-1}{\specl{n}{F}}\to 0.
\]
\end{cor}

\begin{lem}
For all $n\geq 2$, the map $\T{n}$ induces a surjective map $\fil{1}{n}\to\mwk{n}{F}$.
\end{lem}
\begin{proof}
First observe that since $\mwk{n}{F}$ is generated by the elements of the form $[a_1]\cdots[a_n]$ 
it follows from the definition of $\T{n}$ that $\T{n}:\SSn{F^n}\to\mwk{n}{F}$ is surjective for all $n\geq 1$.

Next, since $\mwk{\bullet}{F}$ is multiplicative, $\T{\bullet}$ factors through an algebra homomorphism 
$\mul{\SSn{F^\bullet}}\to\mwk{\bullet}{F}$.
The lemma thus follows from  Corollary \ref{cor:E} and the fact that $\T{2}(E)=0$.
\end{proof}

\begin{lem}
$\fil{1}{2}=\fil{0}{2}$ and $\T{2}:\fil{1}{2}\to\mwk{2}{F}$ is an isomorphism.
\end{lem}
\begin{proof} Since $\hoz{1}{\specl{1}{F}}=0$, $\fil{1}{2}=\fil{0}{2}=E^\infty_{2,0}=\ee{2}(\hoz{2}{\specl{2}{F}})$. 
Now apply Theorem \ref{thm:T2}.
\end{proof}

It is natural to define elements $\ksp{a,b}\in \fil{0}{2}\subset\mul{\SSn{F^2}}$ by $\ksp{a,b}:=\T{2}^{-1}([a][b])$.

\begin{lem} \label{lem:ab}
In $\mul{\SSn{F^2}}$ we have the formula
\[
\ksp{a,b}=\ssb{a}\ssprod\ssb{b}-\ffist{a}\ffist{b}E.
\]
\end{lem}
\begin{proof}
The results above show that the maps $\T{2}$ and $\Dd{2}$ induce an isomorphism 
\[
(\T{2},\Dd{2}):\mul{\SSn{F^2}}\cong\mwk{2}{F}\oplus \grf{F}.
\]
Since $\Dd{2}(\ssb{a}\ssprod\ssb{b})=\ffist{a}\ffist{b}$, while $\Dd{2}(E)=1$, the result follows.
\end{proof}
\begin{thm} \label{thm:star}
\ 
\begin{enumerate}
\item The product $\ssprod$ respects the filtrations on $\SSn{F^n}$; i.e. for all $n,m\geq 1$ and $i,j\geq 0$
\[
\fil{i}{n}\ssprod\fil{j}{m}\subset \fil{i+j}{n+m}.
\]
\item For $n\geq 1$, let $\epsilon_{n+1,1}$ denote the composite $\fil{1}{n+1}\to E^{\infty}_{n,1}=E^2_{n,1}\to \hoz{n}{\specl{n}{F}}$.
For all $a\in F^\times$ and for all $n\geq 1$ the following diagram commutes:
\[
\xymatrix{\fil{0}{n}\ar[r]^-{\ssb{a}\ssprod}&\fil{1}{n+1}\ar[d]^-{\epsilon_{n+1,1}}\\
\hoz{n}{\specl{n}{F}}\ar[u]^-{\ee{n}}\ar[r]^-{\ffist{a}\cdot}&\hoz{n}{\specl{n}{F}}\\}
\] 
\end{enumerate}
\end{thm}
\begin{proof}\ 
\begin{enumerate}
\item The filtration on $\mul{\SSn{F^n}}$ is derived from the spectral sequence $\EEE{F^n}$. This is the spectral sequence of the 
double complex $B_\bullet\otimes_{\specl{n}{F}}\compt{\bullet}{F^n}$, regarded as a filtered complex by truncating $\compt{\bullet}{F^n}$ 
at $i$ for $i=0,1,\ldots$. Since the product $\ssprod$ is derived from a graded bilinear pairing on the complexes $\compt{\bullet}{F^n}$, 
the result easily follows.
\item The spectral sequence $\EEE{F^{n+1}}$ calculates 
\[
\ho{\bullet}{\specl{n+1}{F}}{\compt{}{F^{n+1}}}\cong 
\ho{\bullet}{\specl{n+1}{F}}{H(F^{n+1})}[n+1]
\]
 (where $[n+1]$ denotes a degree shift by $n+1$).

Let $C[1,n]$ denote the truncated complex
\[
\xymatrix{
\compt{1}{F^{n+1}}\ar[r]^-{d_1}&\compt{0}{F^{n+1}}\\
}
\]
and let $Z_1$ denote the kernel of $d_1$.  Then
\[
\ho{\bullet}{\specl{n+1}{F}}{C[1,n]}\cong\ho{\bullet}{\specl{n+1}{F}}{Z_1}[1].
\]
If $\mathcal{F}_i$ denotes the filtration on $\ho{\bullet}{\specl{n+1}{F}}{\compt{}{F^{n+1}}}$ associated to 
 the spectral sequence $\EEE{F^{n+1}}$, then from the definition of this filtration
\[
\image{\ho{k}{\specl{n+1}{F}}{C[1,n]}\to\ho{k}{\specl{n+1}{F}}{\compt{}{F^{n+1}}}}=\mathcal{F}_1\ho{k}{\specl{n+1}{F}}{\compt{}{F^{n+1}}}.
\]
In particular, 
\[
\fil{1}{n+1}\cong \image{\ho{n+1}{\specl{n+1}{F}}{C[1,n]}\to\ho{n+1}{\specl{n+1}{F}}{\compt{}{F^{n+1}}}}
\]
and with this identification the diagram
\[
\xymatrix{
\ho{n}{\specl{n+1}{F}}{Z_1}\ar[r]^-{\cong}\ar[drr]&\ho{n+1}{\specl{n+1}{F}}{C[1,n]}\ar[r]&\fil{1}{n+1}\ar[d]^-{\epsilon_{n+1,1}}\\
&&\ho{n}{\specl{n+1}{F}}{\compt{1}{F^{n+1}}}\\
}
\]
commutes (and  $\ho{n}{\specl{n+1}{F}}{\compt{1}{F^{n+1}}}\cong\hoz{n}{\saff{F}{F^n}}$ by Shapiro's Lemma, of course).

We consider $\specl{n}{F}\subset \saff{F}{F^n}\subset 
\specl{n+1}{F}\subset \genl{n+1}{F}$ where the first inclusion is obtained by inserting a $1$ in the $(1,1)$ 
position. Let $B_\bullet$ denote a projective resolution of $\Z$ over $\gr{\Z}{\genl{n+1}{F}}$. Let $z\in \hoz{n}{\specl{n}{F}}$ be 
represented by $x\otimes 1\in B_n\otimes_{\gr{\Z}{\specl{n}{F}}}\Z = B_n\otimes_{\gr{\Z}{\specl{n}{F}}}\compt{0}{F^n}$. Then 
$\ssb{a}\ssprod\ee{n}(z)$ is represented by $z\otimes [(ae_1)-(e_1)]\in B_n\otimes_{\specl{n+1}{F}}Z_1$ which maps to the element of
$\ho{n}{\specl{n+1}{F}}{\compt{1}{F^{n+1}}}$ represented by $z(g-1)\otimes (e_1)$ where $g=\mathrm{diag}(a,1,\ldots,1,a^{-1})$. But this is 
just the image of $\ffist{a}z$ under the map $\hoz{n}{\specl{n}{F}}\to\hoz{n}{\saff{F}{F^n}}\cong 
\ho{n}{\specl{n+1}{F}}{\compt{1}{F^{n+1}}}$. 
\end{enumerate}
\end{proof}
\begin{lem}\label{lem:f13} The map $\T{3}:\fil{1}{3}\to \mwk{3}{F}$ is an isomorphism.
\end{lem}
\begin{proof}
Consider the short exact sequence 
\[
0\to E^\infty_{3,0}\to \fil{1}{3}\to E^\infty_{2,1}\to 0.
\]
Here $\ee{3}$ induces an isomorphism 
\[
E^\infty_{3,0}\cong \coker{\hoz{3}{\specl{2}{F}}\to \hoz{3}{\specl{3}{F}}}.
\] 
By the main result of \cite{hutchinson:tao2} (Theorem 4.7 - see also section \ref{sec:hommwk} of this article), 
$\T{3}$ thus induces an isomorphism
$E^\infty_{3,0}\cong 2\milk{3}{F}\subset \mwk{3}{F}$.

On the other hand, 
\[
E^\infty_{2,1}\cong \ker{\hoz{2}{\specl{2}{F}}\to\hoz{2}{\specl{3}{F}}}\cong I^3(F)
\] 

Thus we have a commutative diagram
\[
\xymatrix{
0\ar[r]
&E^\infty_{3,0}\ar[d]^-{\T{3}}_-{\cong}\ar[r]
&\fil{1}{3}\ar[d]^-{\T{3}}\ar[r]^-{\rho}
&I^3(F)\ar[d]^-{\alpha}\ar[r]
&0\\
0\ar[r]
&2\milk{3}{F}\ar[r]
&\mwk{3}{F}\ar[r]^-{p_3}
&I^3(F)\ar[r]
&0
}
\]
where the vertical arrows are surjections. 

Now the inclusion $I^3(F)\to \mwk{2}{F}$ is given by $\pfist{a,b,c}\mapsto \ffist{a}[b][c]$. Thus the inclusion 
$j:I^3(F)\to\hoz{2}{\specl{2}{F}}$ is given by $\pfist{a,b,c}\mapsto \ffist{a}\hsp{b,c}$ where $\hsp{b,c}=\ee{2}^{-1}(\ksp{b,c})$.
Thus for all $a,b,c\in F^\times$ we have 
\[
j\circ\rho(\ssb{a}\ssprod\ksp{b,c})=\epsilon_{3,1}(\ssb{a}\ssprod\ksp{b,c})=\ffist{a}\hsp{b,c}
\]
using Theorem \ref{thm:star} (2), and thus $\rho(\ssb{a}\ssprod\ksp{b,c})=\pfist{a,b,c}\in I^3(F)$.  It follows from the diagram 
that 
\[
\alpha(\pfist{a,b,c})=\alpha\circ\rho(\ssb{a}\ssprod\ksp{b,c})=p_3\circ \T{3}(\ssb{a}\ssprod\ksp{b,c})=\pfist{a,b,c}
\]
so that $\alpha$ is the identity map, and the result follows.
\end{proof}
\begin{lem}
For all $a\in F^\times$, $\ssb{a}\ssprod E = E \ssprod \ssb{a}$ in 
$\mul{\SSn{F^3}}$.
\end{lem}
\begin{proof}
By the calculations above, $\fil{1}{3}=\mul{\SSp{F^3}}=\ker{\Dd{3}}$. Thus \\
$R_a:= \ssb{a}\ssprod E- E \ssprod \ssb{a}\in \fil{1}{3}$. But then $\T{3}(R_a)=0 $ since $\T{2}(E)=0$ and thus $R_a=0$ by 
the previous lemma.
\end{proof}
%\begin{rem}
%The results below will show that the map $\alpha$  above (and hence $\T{3}$ also) is 
%an isomorphism. A direct proof of this fact would shorten our argument.
%\end{rem}

\begin{lem}\label{lem:identities}\  

\begin{enumerate}
\item For all $a,b,c\in F^\times$ 
\[
\ssb{a}\ssprod\ksp{b,c}=\ksp{a,b}\ssprod\ssb{c}\mbox{ in } \mul{\SSn{F^3}}.
\]
\item
For all $a,b,c\in F^\times$ 
\[
\ssb{a}\ssprod\ssb{b}\ssprod\ssb{c} = \ssb{c}\ssprod\ssb{a}\ssprod\ssb{b}\mbox{ in }
\mul{\SSn{F^3}}.
\]
\item For all $a,b,c,d\in F^\times$
\[
\ksp{a,b}\ssprod\ksp{c,d}=\ksp{a,c^{-1}}\ssprod\ksp{b,d}\mbox{ in }\mul{\SSn{F^4}}.
\]
\end{enumerate}
\end{lem}

\begin{proof}
The calculations above have established that the map 
\[
(\T{3},\Dd{3}):\mul{\SSn{F^3}}\to \mwk{3}{F}\oplus\aug{F^\times}
\]
is an isomorphism.
\begin{enumerate}
\item This follows from the identities
\[
\T{3}(\ssb{a}\ssprod\ksp{b,c})=[a][b][c]=\T{3}(\ksp{a,b}\ssprod\ssb{c})\mbox{ and } 
\Dd{3}(\ssb{a}\ssprod\ksp{b,c})=\ffist{a,b,c}=\Dd{3}(\ksp{a,b}\ssprod\ssb{c})
\]

\item 
This follows from the fact that $[a][b][c]=[c][a][b]$ in $\mwk{3}{F}$.
\item 
We begin by observing that, since $\SSn{F}\cong\aug{F^\times}$ as a $\grf{F}$-module
we have $\ffist{a}\ssb{b}=\ssb{ab}-\ssb{a}-\ssb{b}=\ffist{b}\ssb{a}$ 
for all $a,b\in F^\times$.

For $x_1,\ldots,x_n\in F^\times$ and $i,j\geq 1$ with $i+j=n$ we set 
\[
L_{i,j}(x_1,\ldots,x_n):=\ffist{x_1}\cdots\ffist{x_i}\left(\ssb{x_{i+1}}\ssprod\cdots\ssprod
\ssb{x_n}\right)\in \mul{\SSn{F^j}}.
\]
 By the observation just made,  we have
\[
L_{i,j}(x_1,\ldots,x_n)=L_{i,j}(x_{\sigma(1)},\ldots,x_{\sigma(n)})
\]
for any permutation $\sigma$ of $1,\ldots,n$. 

So 
\begin{eqnarray*}
\ksp{a,b}\ssprod\ksp{c,d}=
\left(\ssb{a}\ssprod\ssb{b}-\ffist{a}\ffist{b}E\right)\ssprod\left(\ssb{c}\ssprod\ssb{d}
-\ffist{c}\ffist{d}E\right)\\
=\ssb{a}\ssprod\ssb{b}\ssprod\ssb{c}\ssprod\ssb{d}-2L_{2,2}(a,b,c,d)\ssprod E 
+ \ffist{a}\ffist{b}\ffist{c}\ffist{d}E^{\ssprod 2}
\end{eqnarray*}
Let $R=\ksp{a,b}\ssprod\ksp{c,d}-\ksp{a,c^{-1}}\ssprod\ksp{b,d}$.

So $R=$
\begin{eqnarray*}
\ssb{a}\ssprod\ssb{b}\ssprod\ssb{c}\ssprod\ssb{d}-
\ssb{a}\ssprod\ssb{c^{-1}}\ssprod\ssb{b}\ssprod\ssb{d}
-2(L_{2,2}(a,b,c,d)-L_{2,2}(a,c^{-1},b,d))\ssprod E\\ +
\ffist{a}\ffist{d}\left[(\ffist{b}\ffist{c}-\ffist{c^{-1}}\ffist{b})E\right]\ssprod E.
\end{eqnarray*}
However, since $\ksp{b,c}=\ksp{c^{-1},b}$ in $\mul{\SSn{F^2}}$ we have (by Lemma
\ref{lem:ab})
\[
(\ffist{b}\ffist{c}-\ffist{c^{-1}}\ffist{b})E=\ssb{b}\ssprod\ssb{c}-\ssb{c^{-1}}\ssprod\ssb{b}.
\]
Thus
\[
\ffist{a}\ffist{d}\left[(\ffist{b}\ffist{c}-\ffist{c^{-1}}\ffist{b})E\right]\ssprod E =
(L_{2,2}(a,b,c,d)-L_{2,2}(a,c^{-1},b,d))\ssprod E
\]
and hence $R=$
\[
\ssb{a}\ssprod\ssb{b}\ssprod\ssb{c}\ssprod\ssb{d}-
\ssb{a}\ssprod\ssb{c^{-1}}\ssprod\ssb{b}\ssprod\ssb{d}
-(L_{2,2}(a,b,c,d)-L_{2,2}(a,c^{-1},b,d))\ssprod E.
\]
Now 
\begin{eqnarray*}
(L_{2,2}(a,b,c,d)-L_{2,2}(a,c^{-1},b,d))\ssprod E&=&
\ssb{a}\ssprod\ssb{d}\ssprod\left[ (\ffist{b}\ffist{c}-\ffist{c^{-1}}\ffist{b})E\right]\\
&=&\ssb{a}\ssprod\ssb{d}\ssprod\left[ \ssb{b}\ssprod\ssb{c}-\ssb{c^{-1}}\ssprod\ssb{b}
\right]\\
&=&\ssb{a}\ssprod(\ssb{d}\ssprod\ssb{b}\ssprod\ssb{c})-
\ssb{a}\ssprod(\ssb{d}\ssprod\ssb{c^{-1}}\ssprod\ssb{b})\\
&=&\ssb{a}\ssprod\ssb{b}\ssprod\ssb{c}\ssprod\ssb{d}-
\ssb{a}\ssprod\ssb{c^{-1}}\ssprod\ssb{b}\ssprod\ssb{d}
\end{eqnarray*}
using (2) in the last step.
\end{enumerate}
\end{proof}
\begin{thm}\label{thm:mm}
For all $n\geq 2$ there is a homomorphism $\mm{n}:\mwk{n}{F}\to\fil{1}{n}$ such that 
the composite $\T{n}\circ\mm{n}$ is the identity map.
\end{thm}
\begin{proof}
For $n\geq 2$ and $a_1,\ldots,a_n\in F^\times$, let 
\[
\ksb{a_1,\ldots,a_n}:=
\left\{
\begin{array}{ll}
\ksp{a_1,a_2}\ssprod\cdots\ssprod\ksp{a_{n-1},a_n},& n \mbox{ even }\\
\ssb{a_1}\ssprod\ksp{a_2,a_3}\ssprod\cdots\ssprod\ksp{a_{n-1},a_n},& n \mbox{ odd }
\end{array}
\right\}
\in\fil{1}{n}\subset \mul{\SSn{F^n}}.
\]
By Lemma \ref{lem:identities} (1) and (3), as well as the definition of $\ksp{x,y}$, 
the elements $\ksb{a_1,\ldots,a_n}$ satisfy the `Matsumoto-Moore' relations (see Section \ref{sec:hommwk} above), 
and thus there is a well-defined homomorphism 
of groups
\[
\mm{n}:\mwk{n}{F}\to \fil{1}{n}, \qquad[a_1]\cdots[a_n]\mapsto \ksb{a_1,\ldots,a_n}.
\] 
Since $\T{n}(\ksb{a_1,\ldots,a_n})=[a_1]\cdots[a_n]$, the result follows.
\end{proof}

\begin{cor}
The subalgebra of $\SH{2\bullet}{F}$ generated by $\SH{2}{F}=
\hoz{2}{\specl{2}{F}}$ is isomorphic to $\mwk{2\bullet}{F}$ and is  a direct summand 
of $\SH{2\bullet}{F}$.
\end{cor}
\begin{proof}
This is immediate from Theorems  \ref{thm:T2} and \ref{thm:mm}.
\end{proof}

\section{Decomposabilty}\label{sec:decomp}
\textit{Recall that $F$ is a field  of characteristic $0$ throughout this section.}

In \cite{sus:tors}, Suslin proved that $\hoz{n}{\genl{n}{F}}/\hoz{n}{\genl{n-1}{F}}\cong \milk{n}{F}$. This is, in particular,
 a decomposability 
result. It says that $\hoz{n}{\genl{n}{F}}$ is generated, modulo the image of $\hoz{n}{\genl{n-1}{F}}$ by products of $1$-dimensional 
cycles. In this section we will prove analogous results for the special linear group, with Milnor-Witt $K$-theory replacing 
Milnor $K$-theory. To do 
this, we prove the decomposability of the algebra $\mul{\SSn{F^\bullet}}$ (for $n\geq 3$). Theorem \ref{thm:dec} is an analogue of 
Suslin's Proposition 3.3.1. The proof is essentially identical, and we reproduce it here for the convenience of the reader. From this 
we deduce our decomposability result (Theorem \ref{thm:ind}), which requires still 
a little more work than in the case of the general linear group.

\begin{lem}
For any finite-dimensional vector spaces $W$ and $V$, the image of the pairing 
\begin{eqnarray}\label{eqn:mul}
\SSS(W,V)\otimes H(W)\to \mul{\SSS(W\oplus V)}
\end{eqnarray}
coincides with the image of the pairing
\begin{eqnarray}\label{eqn:mul2}
\SSS(V)\otimes \SSS(W)\to \mul{\SSS(W\oplus V)} 
\end{eqnarray}
\end{lem}

\begin{proof}
The image of the pairing (\ref{eqn:mul}) is equal to the image of 
\[
\mul{\SSS(W,V)}\otimes H(W)\to \mul{\SSS(W\oplus V)}
\]
which coincides with the image of
\[
\mul{\SSS(V)}\otimes\mul{\SSS(W)}\to \mul{\SSS(W\oplus V)}
\] 
by the isomorphism of Corollary \ref{cor:isom}.
\end{proof}

Let $\SSd{F^n}\subset \mul{\SSS(F^n)}$ be \emph{the $\grf{F}$-submodule of decomposable elements}; i.e.
  $\SSd{F^n}$ is the image of 
\[
\xymatrix{
\bigoplus_{p+q=n, p,q >0}\left( \mul{\SSS(F^p)}\otimes\mul{\SSS(F^q)}\right)\ar[r]^-{\ssprod}&\mul{\SSS(F^n)}.
}
\]

More generally, note that if $V=V_1\oplus V_2=V'_1\oplus V'_2$ and if $\dim{F}{V_i}=\dim{F}{V'_i}$ for $i=1,2$, then the image of 
$\SSS(V_1)\otimes\SSS(V_2)\to \SSS(V)$ coincides with $\SSS(V'_1)\otimes\SSS(V'_2)\to \SSS(V)$. This follows from the fact that there 
exists $\phi\in\spcl{V}$ with $\phi(V_i)=V'_i$ for $i=1,2$.  

Therefore $\SSd{F^n}$ is the image of 

\[
\xymatrix{
\bigoplus_{F^n=V_1\oplus V_2, V_i\not= 0}\left( \mul{\SSS(V_1)}\otimes\mul{\SSS(V_2)}\right)\ar[r]^-{\ssprod}&\mul{\SSS(F^n)}.
}
\]

If $x=\sum_in_i(x^i_1,\ldots,x^i_p)\in \cgen{p}{V}$ and $y=\sum_jm_j(y^j_1,\ldots,y^j_q)\in\cgen{q}{V}$ and if 
$(x^i_1,\ldots,x^i_p,y^j_1,\ldots,y^j_q)\in\xgen{p+q}{V}$ for all $i,j$, then we let 
\[
x\psprod y:= \sum_{i,j}n_im_j(x^i_1,\ldots,x^i_p,y^j_1,\ldots,y^j_q)\in\cgen{p+q}{V}.
\]

Of course, if $x\in \cgen{p}{V_1}$ and $y\in \cgen{q}{V_2}$ with $V=V_1\oplus V_2$, then $x\psprod y=x\ssprod y$. Furthermore,
when $x\psprod y$ is defined, we have 
\[
d(x\psprod y)=d(x)\psprod y +(-1)^px\psprod d(y).
\] 

\begin{thm}\label{thm:dec}
Let $n\geq 1$. For any $a_1,\ldots,a_n,b\in F^\times$ and for any $1\leq i\leq n$
\[
\ssb{a_1,\ldots,ba_i,\ldots,a_n}\cong\fgen{b}\ssb{a_1,\ldots,a_n}\pmod{\SSd{F^n}}.
\]
\end{thm}
\begin{proof} Let $a=a_1e_1+\cdots + ba_ie_i+\cdots a_ne_n$. 

We have
\begin{eqnarray*}
\ssb{a_1,\ldots,ba_i,\ldots,a_n}-\fgen{b}\ssb{a_1,\ldots,a_n}&=&
 d(e_1,\ldots,e_i,\ldots,e_n,a)-d(e_1,\ldots,b_ie_i,\ldots,e_n,a)\\
&=&d\bigg((e_1,\ldots,e_{i-1})\psprod \left( (e_i)-(be_i)\right)\psprod (e_{i+1},\ldots,e_n,a)\bigg)\\
&=&d(e_1,\ldots,e_{i-1})\psprod \left( (e_i)-(be_i)\right)\psprod (e_{i+1},\ldots,e_n,a)\\
&+&(-1)^i(e_1,\ldots,e_{i-1})\psprod \left( (e_i)-(be_i)\right)\psprod d(e_{i+1},\ldots,e_n,a)
\end{eqnarray*}

Let $u=a_1e_1+\cdots+a_{i-1}e_{i-1}+ba_ie_i=a-\sum_{j=i+1}^na_je_j$.Then 
\[
(-1)^{i-1}(e_1,\ldots,e_{i-1})=d\big((e_1,\ldots,e_{i-1})\psprod (u)\big) - d(e_1,\ldots,e_{i-1})\psprod (u)
\]
and
\[
(e_{i+1},\ldots,e_n,a)=d\big((u)\psprod(e_{i+1},\ldots,e_n,a)\big)+(u)\psprod d(e_{i+1},\ldots,e_n,a). 
\]
Thus $\ssb{a_1,\ldots,ba_i,\ldots,a_n}-\fgen{b}\ssb{a_1,\ldots,a_n}=X_1-X_2+X_3$ where 
\begin{eqnarray*}
X_1 &=& d(e_1,\ldots,e_{i-1})\psprod\big((e_i)-(be_i)\big)\psprod d(u,e_{i+1},\ldots,e_n,a),\\
X_2&=&d(e_1,\ldots,e_{i-1},u)\psprod\big((e_i)-(be_i)\big)\psprod d(e_{i+1},\ldots,e_n,a),\mbox{ and }\\
X_3&=&d(e_1,\ldots,e_{i-1})\psprod \bigg[\big((e_i)-(be_i)\big)\psprod (u)+(u)\psprod \big((e_i)-(be_i)\big)\bigg]\psprod
d(e_{i+1},\ldots,e_n,a)\\
\end{eqnarray*}

We show that each $X_i$ is decomposable: Let $V\subset F^n$ be the span of $u, e_{i+1},\ldots,e_n$ (which is also equal to the 
span of $a,e_{i+1},\ldots, e_n$), and let $V'$ be the span of $e_1,\ldots,e_{i-1}$. Then $F^n=V'\oplus V$ and 
$d(u,e_{i+1},\ldots,e_n,a)\in H(V)$ while\\
  $d(e_1,\ldots,e_{i-1})\psprod\big((e_i)-(be_i)\big)\in H(V,V')$. 

Thus $X_1$ lies in 
the image of 
\[
\xymatrix{
H(V,V')\otimes H(V)\ar[r]^-{\ssprod} &\mul{\SSn{F^n}}
}
\]
and so is decomposable.

Similarly, if we let $W$ be the span of $e_1,\ldots,e_i$ and $W'$ the span of $e_{i+1},\ldots,e_n$, then 
\[
d(e_1,\ldots,e_{i-1},u)\psprod\big((e_i)-(be_i)\big),\ d(e_1,\ldots,e_{i-1})\psprod
 \bigg[\big((e_i)-(be_i)\big)\psprod (u)+(u)\psprod \big((e_i)-(be_i)\big)\bigg]\in H(W)
\]
and $d(e_{i+1},\ldots,e_n,a)\in H(W,W')$.Thus $X_2, X_3$ lie in 
the image of 
\[
\xymatrix{
H(W)\otimes H(W,W')\ar[r]^-{\ssprod} &\mul{\SSn{F^n}}
}
\]
and are also decomposable. 
\end{proof}

Let $\SSind{F^n}:=\mul{\SSn{F^n}}/\SSd{F^n}$.

The main goal of this section is to show that $\SSind{F^n}=0$ for all $n\geq 3$ (Theorem \ref{thm:ind} below).

\begin{lem} 
For all $n\geq 3$, $\SSind{F^n}$ is a \mult $\grf{F}$-module.
\end{lem}
\begin{proof}
We have
\[
\Aa{n}\cong
\left\{
\begin{array}{ll}
\grf{F}E^{\ssprod n/2},&n \mbox{ even}\\
\SSn{F}\ssprod E^{\ssprod (n-1)/2},&n \mbox{ odd}
\end{array}
\right.
\]
and these modules are  decomposable for all $n\geq 3$. It follows that the map 
\[
\mul{\SSp{F^n}}\to \SSind{F^n}
\]
is surjective for all $n\geq 3$.
\end{proof}
\begin{rem}
Since $E\ssprod \mul{\SSn{F^{n-2}}}\subset \SSd{F^n}$, in fact we have that $\fil{1}{n}\to \SSind{F^n}$ is surjective.
\end{rem}

Theorem \ref{thm:dec} shows that for all $a_1,\ldots,a_n\in F^\times$
\[
\ssb{a_1,\ldots,a_n}\cong \fgen{\prod_ia_i}\ssb{1,\ldots,1}\pmod{\SSd{F^n}}.
\]

In other words the map 
\[
\grf{F}\to \SSind{F^n},\quad \alpha\mapsto \alpha\ssb{1,\ldots, 1}
\]
is a surjective homomorphism of $\grf{F}$-modules. Thus, we are required to establish that 
$\ssb{1,\ldots,1}\in \SSd{F^n}$ for all $n\geq 3$.

For convenience below, we will let $\SSf{n}{F}$ denote the free $\grf{F}$-module on the symbols $\ssf{a_1,\ldots,a_n}$, 
 $a_1,\ldots,a_n\in F^\times$. Let $p_n:\SSf{n}{F}\to\SSn{F^n}$ be the $\grf{F}$-module homomorphism sending 
$\ssf{a_1,\ldots,a_n}$ to $\ssb{a_1,\ldots,a_n}$. We will say that $\sigma\in \SSn{F^n}$ is represented by 
$\tilde{\sigma}\in \SSf{n}{F}$ if $p_n(\tilde{\sigma})=\sigma$.

Note that $\SSf{\bullet}{F}$ can be given the structure of a graded $\grf{F}$-algebra by setting 
\[
\ssf{a_1,\ldots,a_n}\cdot\ssf{a_{n+1},\ldots,a_{n+m}}:=\ssf{a_1,\ldots,a_{n+m}}; 
\]
i.e., we can identify $\SSf{\bullet}{F}$ with the tensor algebra over $\grf{F}$ on the free module with basis $\ssf{a}$, $a\in F^\times$.

Let $\Pi_\bullet:\SSf{\bullet}{F}\to \grf{F}[x]$ be the homomorphism of graded $\grf{F}$-algebras sending 
$\ssf{a}$ to $\fgen{a}x$.

For all $n\geq 1$ we have a commutative square of surjective homomorphisms of $\grf{F}$-modules
\[
\xymatrix{
\SSf{n}{F}\ar[r]^-{\Pi_n}\ar[d]^-{p_n}&\grf{F}\cdot x^n\ar[d]^-{\gamma_n}\\
\SSn{F^n}\ar[r]&\SSind{F^n}\\
}
\] 
where $\gamma_n(x^n)=\ssb{1,\ldots,1}$.

\begin{lem} \label{lem:odd}
If $n$ is odd and $n\geq 3$ then $\SSind{F^n}=0$; i.e.,  
\[\mul{\SSn{F^n}}=\SSd{F^n}.\]
\end{lem}
\begin{proof} From the fundamental relation in $\SSn{F^n}$ (Theorem \ref{thm:pres}), if $b_1,\ldots, b_n$ are distinct elements of 
$F^\times$, then  $0\in \SSn{F^n}$ is represented by 
\[
R_b:=\ssf{b_1,\ldots,b_n}-\ssf{1,\ldots,1}
-\sum_{j=1}^n(-1)^{n+j}\fgen{(-1)^{n+j}}\ssf{b_1-b_j,\ldots,\widehat{b_j-b_j},\ldots,b_n-b_j,b_j}\in\SSf{n}{F}.
\] 
Now
\[
\Pi_n(R_b)=\left[\fgen{\prod_i b_i}-\fgen{1}-\sum_{j=1}^n(-1)^{n+j}\fgen{(b_j-b_1)\cdots(b_j-b_{j-1})\cdot(b_{j+1}-b_j)\cdots(b_n-b_j)
\cdot b_j}
\right]x^n.
\]
Now choose $b_i=i$, $i=1,\ldots,n$.
Then
\[
\Pi_n(R_b)=\left[\fgen{n!}-\fgen{1}-\sum_{j=1}^n(-1)^{n+j}\fgen{j!(n-j)!}\right]x^n=-\fgen{1}x^n \mbox{ since $n$ is odd}.
\]
It follows that $-\ssb{1,\ldots,1}=0$ in $\SSind{F^n}$ as required.
\end{proof}

The case $n$ even requires a little more work.

The maps $\{p_n\}_n$ do not define a map of graded algebras. However, we do have the following:
\begin{lem}\label{lem:prod}
For $1\not=a\in F^\times$, let 
\[
L(x):=\fgen{-1}\ssf{1-x,1}-\fgen{x}\ssf{1-\frac{1}{x},\frac{1}{x}}+\ssf{1,1}\in \SSf{2}{F}.
\]

Then for all $a_1,\ldots,a_n\in F^\times\setminus\{ 1\}$, the product
\[
\prod_{i=1}^n\ssb{1,a_i}=\ssb{1,a_1}\ssprod\cdots\ssprod\ssb{1,a_n}\in \SSn{F^{2n}}
\]
is represented by $\prod_iL(a_i)\in \SSf{2n}{F}$.
\end{lem}
\begin{proof}
For convenience of notation, we will represent standard basis elements of $\cgen{q}{F^n}$ as $n\times q$ matrices $[v_1|\cdots|v_q]$.

Let $e=(1,\ldots,1)$ and  let $\sigma_i(C)$ denote the sum of the entries in the $i$th row of the 
$n\times n$ matrix $C$. By Remark \ref{rem:ssn}, if $A\in\genl{n}{F}$ and $[A|e]\in \xgen{n+1}{F^n}$ then
 $d_{n+1}([A|e])$ represents $\fgen{\det{A}}\ssb{\sigma_1(A^{-1}),\ldots,\sigma_n(A^{-1})}\in \SSn{F^n}$.

Now, for $a\not= 1$, $\ssb{1,a}$ is represented in $\SSn{F^2}$ by
\[
d_3\left(
\begin{bmatrix}
1&0&1\\
0&1&a
\end{bmatrix}
\right)
=
\begin{bmatrix}
0&1\\
1&a
\end{bmatrix}
-
\begin{bmatrix}
1&1\\
0&a
\end{bmatrix}
+
\begin{bmatrix}
1&0\\
0&1
\end{bmatrix}
=T_1(a)-T_2(a)+T_3(a)\in \cgen{2}{F^2}.
\]

From the definition of the product $\ssprod$, it follows that $\ssb{1,a_1}\ssprod\cdots\ssprod\ssb{1,a_n}$ is represented by 
\[
Z:=\sum_{j=(j_1,\ldots,j_n)\in (1,2,3)^n}(-1)^{k(j)}
\begin{bmatrix}
T_{j_1}(a_1)&&\\
&\ddots&\\
&&T_{j_n}(a_n)
\end{bmatrix}
=\sum_j(-1)^{k(j)}T(j,a).
\]
where $k(j):=\card{\{ i\leq n | j_i=2\}}$

Since $a_i\not= 1$ for all $i$, the vector $e=(1,\ldots,1)$ is in general position with respect to the columns of all these matrices.
Thus we can use the partial homotopy operator $s_e$ to write this cycle as a boundary:
\[
Z=\sum_j(-1)^{k(j)}d_{2n+1}\left([T(j,a)|e] \right).
\]

By the remarks above 
\[
d_{2n+1}\left( [T(j,a)|e] \right) = \fgen{\prod_i\det{T_{j_i}(a_i)}}
\ssb{\sigma_{1}(T_{j_1}(a_1)),\sigma_{2}(T_{j_1}(a_1)),\sigma_{1}(T_{j_2}(a_2)),\ldots,\sigma_{1}(T_{j_n}(a_n)),\sigma_{2}(T_{j_n}(a_n))}.
\]
This is represented by 
\begin{eqnarray*}
\fgen{\prod_i\det{T_{j_i}(a_i)}}
\ssf{\sigma_{1}(T_{j_1}(a_1)),\sigma_{2}(T_{j_1}(a_1)),\sigma_{1}(T_{j_2}(a_2)),\ldots,\sigma_{1}(T_{j_n}(a_n)),\sigma_{2}(T_{j_n}(a_n))}\\
= \prod_{i=1}^n\bigg( \fgen{\det{T_{j_i}(a_i)}}\ssf{\sigma_{1}(T_{j_i}(a_i)),\sigma_{2}(T_{j_i}(a_i))}\bigg)\in \SSf{2n}{F}.
\end{eqnarray*}
Thus $Z$ is represented by 
\begin{eqnarray*}
\sum_j(-1)^{k(j)}\prod_{i=1}^n\bigg( \fgen{\det{T_{j_i}(a_i)}}\ssf{\sigma_{1}(T_{j_i}(a_i)),\sigma_{2}(T_{j_i}(a_i))}\bigg)\\
=\prod_{i=1}^n\bigg(\sum_{j=1}^3(-1)^{j+1}\fgen{\det{T_j(a_i)}}\ssf{\sigma_{1}(T_{j}(a_i)),\sigma_{2}(T_{j}(a_i))}\bigg)
=\prod_{i=1}^nL(a_i)\in \SSf{2n}{F}.
\end{eqnarray*}
\end{proof}

Observe that all of our \mult modules (and in particular $\mul{\SSn{F^n}}$) have the 
following property: they admit a finite filtration $0=M_0\subset M_1\subset \cdots \subset M_t=M$ such that each of the associated 
quotients $M_{r}/M_{r-1}$ is annihilated by $\aug{(F^\times)^{k_r}}$ for some $k_r\geq 1$. From this observation it easily follows that 
\begin{lem}\label{lem:easy}

\[
\SSind{F^n}=0 \iff \SSind{F^n}/(\aug{(F^\times)^r}\cdot \SSind{F^n})=0\mbox{ for all } r\geq 1.
\]
\end{lem}

\begin{thm} \label{thm:ind}
$\SSind{F^n}=0$ for all $n\geq 3$.
\end{thm}
\begin{proof} 
The case $n$ odd has already been dealt with in Lemma \ref{lem:odd}

For the even case, by Lemma \ref{lem:easy} it will be enough to prove  that for all $r\geq 1$
\[
\gr{\Z}{F^\times/(F^\times)^r}\otimes_{\grf{F}}\SSind{F^n}=0.
\]

Fix $r\geq 1$. If $a\in (F^\times)^r\setminus\{ 1\}$, then 
\[
\Pi_2(L(a))=\left(\fgen{a-1}-\fgen{1-\frac{1}{a}}+\fgen{1}\right)x^2 = \fgen{1}x^2\in \gr{\Z}{F^\times/(F^\times)^r}x^2
\]
since 
\[
1-\frac{1}{a}=\frac{a-1}{a}\equiv a-1\pmod{(F^\times)^r}.
\]

Now let $n>1$ and choose $a_1,\ldots,a_n\in (F^\times)^r\setminus\{ 1\}$. Let $\sigma=\ssb{1,a_1}\ssprod\cdots\ssprod\ssb{1,a_n}
\in \SSn{F^{2n}}$, so that $\sigma\mapsto 0$ in $\SSind{F^{2n}}$. By Lemma \ref{lem:prod}, $\sigma$ is represented by 
$\tilde{\sigma}=\prod_{i=1}^nL(a_i)$ in $\SSf{2n}{F}$ and thus
\[
\Pi_{2n}(\tilde{\sigma})=\prod_{i=1}^n\big(\Pi_2(L(a_i))\big)=\fgen{1}\in \gr{\Z}{F^\times/(F^\times)^r}x^{2n}
\]
so that the image of $\sigma$ in $\gr{\Z}{F^\times/(F^\times)^r}\otimes_{\grf{F}}\SSind{F^{2n}}$ is $1\otimes\ssb{1,\ldots, 1}$. 
This proves the 
theorem.
\end{proof}

\begin{cor}\label{cor:f1n}
For all $n\geq 2$, the map $\T{n}$ induces an isomorphism $\fil{1}{n}\cong\mwk{n}{F}$.
\end{cor}
\begin{proof}
%The case $n=2$ has already been proved. 
Since, by the computations above, $\mul{\SSn{F^2}}=\SSn{F}^{\ssprod 2}+\grf{F}E$  
it follows, using Theorem \ref{thm:ind} and induction on $n$, that $\mul{\SSn{F^\bullet}}$
 is generated as a $\grf{F}$-algebra by $\{\ssb{a}\in\SSn{F}| 1\not= a\in F^\times\}$ 
 and $E$. 
 
 Thus $E$ is central in the algebra $\mul{\SSn{F^\bullet}}$ and  for all $n\geq 2$, 
 \[
 \frac{\mul{\SSn{F^n}}}{E\ssprod \mul{\SSn{F^{n-2}}}}
 \] 
 is generated by the elements of the form $\ssb{a_1}\ssprod\cdots\ssprod\ssb{a_n}$, 
 and hence also by the elements $\ksb{a_1,\ldots,a_n}$ since 
 $\ksp{a,b}\equiv \ssb{a}\ssprod\ssb{b}\pmod{\il{E}}$ for all $a,b\in F^\times$.
 
 Since 
 \[
\fil{1}{n}\cong \frac{\mul{\SSn{F^n}}}{E\ssprod \mul{\SSn{F^{n-2}}}}
 \]
 by Corollary \ref{cor:E}, it follows that $\fil{1}{n}$ is generated by the elements 
 $\ksb{a_1,\ldots,a_n}$, and thus that the homomorphisms $\mm{n}$ of Theorem 
 \ref{thm:mm} are surjective.
\end{proof}

\begin{cor} For all $n\geq 3$,
\[
\mul{\SSn{F^n}}\cong
\left\{
\begin{array}{ll}
 \mwk{n}{F}\oplus \mwk{n-2}{F}\oplus \cdots \oplus \mwk{2}{F}\oplus \grf{F}& n \mbox{ even}\\[5pt]
\mwk{n}{F}\oplus \mwk{n-2}{F}\oplus \cdots \oplus \mwk{3}{F}\oplus \aug{F^\times}& n \mbox{ odd}
\end{array}
\right.
\]
as a $\grf{F}$-module.
\end{cor}

\begin{cor}
\label{cor:even}
 For all even $n\geq 2$ the cokernel of the map 
\[
\hoz{n}{\specl{n-1}{F}}\to\hoz{n}{\specl{n}{F}}
\]
is isomorphic to $\mwk{n}{F}$.
\end{cor}
\begin{proof}
Recall that $\ee{2}$ induces an isomorphism $\hoz{2}{\specl{2}{F}}\cong\fil{1}{2}=\fil{0}{2}$. 
Let $\hsp{a,b}$ denote the generator $\ee{2}^{-1}(\ksp{a,b})$ of 
$\hoz{2}{\specl{2}{F}}$. Then for even $n$ 
\begin{eqnarray*}
\ksb{a_1,\ldots,a_n}&=&\ksp{a_1,a_2}\ssprod\cdots\ssprod\ksp{a_{n-1},a_n}\\
                                &=&\ee{2}(\hsp{a_1,a_2})\ssprod\cdots\ssprod\ee{2}
                                (\hsp{a_{n-1},a_n})\\
                                &=&\ee{n}(\hsp{a_1,a_2}\times\cdots\times\hsp{a_{n-1},a_n})
\end{eqnarray*}
by Lemma \ref{lem:ee} (2).

Since $\fil{1}{n}$ is generated by the elements $\ksb{a_1,\ldots,a_n}$, 
it follows that $\fil{1}{n}= \ee{n}(\hoz{n}{\specl{n}{F}})=E^\infty_{n,0}=\fil{0}{n}$, 
proving the result.
\end{proof}
\begin{cor}\label{cor:odd}
For all odd $n\geq 1$ the maps 
\[
\hoz{n}{\specl{k}{F}}\to\hoz{n}{\specl{k+1}{F}}
\]
are isomorphisms for $k\geq n$.
\end{cor}
\begin{proof}
In view of Corollary \ref{cor:stab}, the only point at issue is the injectivity of 
\[
\hoz{n}{\specl{n}{F}}\to\hoz{n}{\specl{n+1}{F}}.
\]
But the proof of Corollary \ref{cor:even} shows that the term 
\[
\fil{1}{n+1}/E^\infty_{n+1,0}\cong E^\infty_{n,1}=\ker{\hoz{n}{\specl{n}{F}}\to\hoz{n}{\specl{n+1}{F}}}
\]
in the spectral sequence $\mul{\EEEp{F^{n+1}}}$ is zero.
\end{proof}
\begin{cor}\label{cor:last}  If 
$n\geq 3$  is odd, then
\begin{eqnarray*}
\coker{\hoz{n}{\specl{n-1}{F}}\to\hoz{n}{\specl{n}{F}}}\cong 2\milk{n}{F}\\
\ker{\hoz{n-1}{\specl{n-1}{F}}\to\hoz{n-1}{\specl{n}{F}}}\cong I^n(F).
\end{eqnarray*}
\end{cor}
\begin{proof} Since we have already proved this result for $n=3$ above, we will assume that $n\geq 5$ ($n$ odd).

Let $a_1,\ldots,a_n\in F^\times$ and let $z\in\hoz{n-1}{\specl{n-1}{F}}$ satisfy $\ee{n-1}(z)=\ksb{a_2,\ldots,a_n}\in \fil{0}{n-1}
\cong \mwk{n-1}{F}$. Thus $\ksb{a_1,\ldots,a_n}=\ssb{a_1}\ssprod\ee{n-1}(z)$ and hence $\epsilon_{n,1}(\ksb{a_1,\ldots,a_n})=
\ffist{a_1}z$ by Theorem \ref{thm:star} (2). It follows that the diagram
\[
\xymatrix{
\fil{1}{n}\ar[r]^-{\epsilon_{n,1}}\ar[d]^-{\T{n}}_-{\cong}&\hoz{n-1}{\specl{n-1}{F}}\ar[d]^-{\T{n-1}\circ\ee{n-1}}\\
\mwk{n}{F}\ar[r]^-{\eta}&\mwk{n-1}{F}\\
}
\] 
commutes.

Now $\ker{\epsilon_{n,1}}=\image{\ee{n}:\hoz{n}{\specl{n}{F}}\to\fil{1}{n}}$. Since $\image{\ee{3})=\T{3}^{-1}(2\milk{3}{F}}$ and 
$\image{\ee{n-3}}=\fil{1}{n-3}=\T{n-3}^{-1}(\mwk{n-3}{F})$ we have 
\[
\T{n}(\image{\ee{n}})=\image{\T{n}\circ\ee{n}}\supset 2\milk{3}{F}\cdot\mwk{n-3}{F}=2\milk{n}{F}\subset\mwk{n}{F}
\]
(using the fact that $\T{\bullet}$ and $\ee{\bullet}$ are algebra homomorphisms).

Thus we get a commutative diagram
\[
\xymatrix{
\frac{\mwk{n}{F}}{2\milk{n}{F}}\ar[r]^-{\T{n}^{-1}}\ar[d]^-{\eta}_{\cong}
&\frac{\fil{1}{n}}{\ker{\epsilon_{n,1}}}\ar[dl]^-{\T{n-1}\circ\ee{n-1}\circ\epsilon_{n,1}}\\
I^n(F)\\
}
\]
from which it follows that the map $\T{n}^{-1}$ in this diagram is an isomorphism, and hence $\image{\ee{n}}=
\ker{\epsilon_{n,1}}\cong 2\milk{n}{F}$ and $\image{\epsilon_{n,1}}\cong I^n(F)$.
\end{proof}
\section{Acknowledgements}
The work in this article was partially funded by  the Science Foundation Ireland Research Frontiers Programme grant 05/RFP/MAT0022.
\bibliographystyle{plain}
\bibliography{BargeMorel}
\end{document}